\documentclass[sn-mathphys]{sn-jnl}%
\usepackage{amssymb}
\usepackage{bm}
\usepackage{amsmath}
\usepackage{epstopdf} %
\usepackage{mathptmx}      %
\usepackage{subcaption}

\jyear{2021}%

\theoremstyle{thmstyleone}%
\newtheorem{theorem}{Theorem}%
\newtheorem{lemma}[theorem]{Lemma}%

\theoremstyle{thmstyletwo}%
\newtheorem{remark}{Remark}%

\theoremstyle{thmstylethree}%
\newtheorem{definition}{Definition}%

\newcommand{\abs}[1]{\lvert#1\rvert}
\newcommand{\mat}[1]{\bm{#1}}
\newcommand{\operator}[1]{\mathsf{#1}}
\newcommand{\vectorspace}[1]{\mathsf{#1}}
\newcommand{\ewleq}{\le}
\newcommand{\ewgeq}{\ge}
\newcommand{\qfleq}{\preceq}
\newcommand{\qfgeq}{\succeq}
\newcommand{\svdots}{\raisebox{3pt}{\scalebox{.75}{\vdots}}}
\newcommand{\sddots}{\raisebox{3pt}{\scalebox{.75}{$\ddots$}}}

\begin{document}

\title[Convergence of Numerical Integration as a Matrix Approximation]{On the Convergence of Numerical Integration
        as a Finite Matrix Approximation to Multiplication Operator}

\author*[1]{\fnm{Juha} \sur{Sarmavuori}}\email{juha.sarmavuori@aalto.fi}

\author[1]{\fnm{Simo} \sur{S\"arkk\"a}}\email{simo.sarkka@aalto.fi}
\equalcont{These authors contributed equally to this work.}

\affil*[1]{\orgdiv{Department of Electrical Engineering and Automation}, \orgname{Aalto University}, \orgaddress{\street{P.O. Box 12200}, \city{Espoo}, \postcode{FI-00076 AALTO}, \country{Finland}}}

\abstract{We study the convergence of a family of numerical integration methods where
  the numerical integration is formulated as a finite matrix approximation to
  a multiplication operator.
  For bounded functions, convergence has already been established using the theory of strong operator convergence.
  In this article, we consider unbounded functions and domains which pose
  several
  difficulties compared to the bounded case.
  A natural choice of method for this study is the theory of strong
  resolvent convergence which has previously been mostly applied to study
  the convergence of approximations of differential operators.
  The existing theory already includes convergence
  theorems that can be used as proofs as such for a limited class of functions
  and extended for a wider class of functions in terms of function
  growth or discontinuity.
  The extended results apply to all self-adjoint operators, not just
  multiplication
  operators.
  We also show how Jensen's operator inequality can be used to analyse the
  convergence of an improper numerical integral of a function bounded by an
  operator convex function.}



\maketitle

\section{Introduction}
\label{intro}
In this article, the aim is to theoretically analyse a class of numerical integration methods for approximate computation of integrals of the form
\begin{equation}
  I = \int_{\Omega} f(g(\mat{x}))\,w(\mat{x})\,d\mat{x},
  \label{eq:integral}
\end{equation}
$\Omega\subset\mathbb{R}^d$, $w : \Omega \to [0,\infty)$, $g : \Omega \to \mathbb{R}$, and $f : \mathbb{R} \to \mathbb{R}$. Without loss of generality, we assume that
$\int_\Omega w(\mat{x})\,d\mat{x}=1$. As we have recently shown in \cite{Sarmavuori+Sarkka:2019}, numerical integration of \eqref{eq:integral} can be formulated as a finite matrix approximation of the multiplication operator as follows (see also \cite{Gragg:1993,Simon:2009,Velazquez:2008,Bultheel+Cantero+Cruz-Barroso:2015}).

Let $\langle \cdot,\cdot \rangle$ be the inner product 
$\langle \phi,\psi \rangle =
  \int_{\Omega}\overline{\phi(\mat{x})}\,\psi(\mat{x})\,
  w(\mat{x})\,d\mat{x}
  $ which defines the norm
  $\|\phi\|=\sqrt{\langle \phi, \phi \rangle}$
  so that
  $\vectorspace{L}^2_w(\Omega)=\{\phi : \|\phi\|<\infty\}$
  is a complete Hilbert space. Let
 $\operator{M}[g]$ be the multiplication operator of multiplying with
  the real function $g$, that is, $\operator{M}[g]\,\phi=g\,\phi$
  almost everywhere
  for all
  $\phi\in\vectorspace{D}(\operator{M}[g])=\{\phi : \|g\,\phi\|<\infty\}$.
  It then turns out that we can define the operator function $f(\operator{M}[g])$ which allows us to express the integral \eqref{eq:integral} as follows:
\begin{equation}
  \int_{\Omega} f(g(\mat{x}))\,w(\mat{x})\,d\mat{x}=
  \langle 1,f(\operator{M}[g])\,1\rangle.
  \label{eq:integral-inner}
\end{equation}
Let us then consider a finite-dimensional subspace of $\vectorspace{L}^2_w(\Omega)$ spanned by orthonormal functions $1=\phi_0,\phi_1,\phi_2,\ldots,\phi_n$ and find the projection of the multiplication operator on this subspace (please note that the indexing starts from zero)
\begin{equation}
  \label{eq:matrix-elements}
  \left[
    \mat{M}_n[g]
    \right]_{i,j}
  =\langle \phi_i,\operator{M}[g]\,\phi_j\rangle=\langle \phi_i,g\,\phi_j\rangle
  =
  \int_\Omega \overline{\phi_i(\mat{x})}\,g(\mat{x})\,\phi_j(\mat{x})\,
w(\mat{x})\,d\mat{x}.
\end{equation}
The operator function $f(\operator{M}[g])$ appearing in \eqref{eq:integral-inner} can then be approximated via matrix function $f(\mat{M}_n[g])$ which leads to the approximation
\begin{align}
  \label{eq:numint-def}
  \int_{\Omega} f(g(\mat{x}))\,w(\mat{x})\,d\mat{x}=
  \langle 1,f(\operator{M}[g])\,1\rangle
  &\approx
  \left[
    f(\mat{M}_n[g])
    \right]_{0,0},\\
  \label{eq:product-convergence}
  \int_{\Omega} \prod_{i=0}^m f_i(g_i(\mat{x}))\,w(\mat{x})\,d\mat{x}
  =\langle 1,\prod_{i=0}^m f_i(\operator{M}[g_i])\,1\rangle
    &\approx
    \left[
      \prod_{i=0}^m f_i(\mat{M}_n[g_i])
    \right]_{0,0}.
\end{align}
As shown in \cite{Sarmavuori+Sarkka:2019},
when the inside function is $g(x)=x$ and the basis functions
are polynomials, we get the classical Gaussian quadratures.

In \cite{Sarmavuori+Sarkka:2019},
convergence was considered only
for
bounded outside function $f$
and a bounded inside function $g$.
In this article, our goal is to extend the convergence results to unbounded
functions. This allows us to consider improper integrals of functions
that attain infinity at either or both endpoints of the integration interval.
Our primary interest is in integrals, where infinity
is attained at an infinite endpoint or at both endpoints $\pm\infty$.
Proof of convergence for this type of integral is often
considerably more complicated than for bounded functions on bounded
intervals.
See
\cite{Uspensky:1928,Jouravsky:1928,Shohat+Tamarkin:1963,Freud:1971,Bultheel+Diaz-Mendoza+Gonzalez-Vera+Orive:1997,Bultheel+Diaz-Mendoza+Gonzalez-Vera+Orive:2000}
for examples of proofs of convergence on unbounded intervals 
for Gaussian quadratures 
and Gaussian quadratures generalised on rational functions.
Similar results have been proved also purely operator theoretically
for Gaussian quadratures on unbounded intervals for polynomially bounded
functions \cite[Sections~1.8~and~1.13]{Simon:2009}. 
We generalise this operator theoretic approach 
on a larger family of inside functions $g$ 
and basis functions
and under some conditions
to faster than polynomially growing outside functions $f$. 
We also consider improper integrals, where infinity is
attained at a finite endpoint.

As a starting point, we take the strong resolvent convergence
(or generalised strong convergence)
\cite{Kato:1995,Reed+Simon:1981,Segal+Kunze:1978,Weidmann:1980,Simon:2015}
that gives well-known convergence results for continuous functions and
characteristic functions.
We extend these results to Riemann--Stieltjes integrable functions
that are, by definition, bounded but can have infinitely many discontinuities.
We also prove convergence results for a class of linearly bounded functions,
that is, functions that are less than $a+b\,\abs{x}$ for some $a,b>0$.
These results apply to general self-adjoint operators, not just
multiplication operators and therefore also apply for more exotic
integrals that it is possible to formulate as a function of a self-adjoint operator 
\cite{Segal:1965}.

For Gaussian quadratures, the convergence proofs are often based on
an inequality of the following type (e.g.
\cite{Uspensky:1928,Jouravsky:1928,Bultheel+Diaz-Mendoza+Gonzalez-Vera+Orive:2000},
\cite[Chapter~IV]{Shohat+Tamarkin:1963}, 
\cite[Chapter~3, Lemma~1.5]{Freud:1971}):
\begin{equation}
  \sum_{k=0}^n w_k\,x_k^{2\,m}\leq \int_{-\infty}^\infty x^{2\,m}\,w(x)\,dx,
\label{eq:wx_ineq}
\end{equation}
that is, a Gaussian quadrature approximation with weights $w_k$ and nodes $x_k$
never exceeds the true value of the integral for even monomials $x^{2\, m}$.
We do not have this type of inequality in general but in special cases when
all matrix elements are positive. We give examples of such cases.
In those cases, we can still prove convergence for certain fast-growing
classes of functions such as exponentially bounded functions.

A similar inequality to \eqref{eq:wx_ineq} also arises from Jensen's operator inequality when the
outside function $f$ is an operator convex function.
We use this type of inequality to prove convergence of
improper integrals, where infinity of the integrand is attained at a
finite endpoint of the integration interval. This type of singularity
has been studied in, for example, \cite{Atkinson+Chien+Hansen:2021}.

The main contributions of this article are:
\begin{enumerate}
\item
  We use the theory of strong resolvent convergence
  for analysing the convergence of matrix method of numerical integration
  when the inside function is unbounded.
\item
  We extend the theory of convergence for functions of self-adjoint operators
  in a converging sequence from continuous and bounded functions to
  wider classes of discontinuous and unbounded functions.
\item
  We show theoretically and by a numerical example that convergence holds
  for a wider class of unbounded functions when the matrix elements are
  nonnegative.
\item
  We also show the convergence of improper integrals on a finite endpoint for
  functions bounded by an operator convex function by using Jensen's operator
  inequality.
\end{enumerate}

The structure of the article is as follows: the key concepts are introduced
in Section~\ref{sec:prelim}, the main results are presented in 
Section~\ref{sec:convergence}, numerical examples are presented
in Section~\ref{sec:numerical-example}, and the conclusions are drawn in 
Section~\ref{sec:conclusions}.

\section{Preliminaries}
\label{sec:prelim}

A self-adjoint operator is such operator that $\operator{A}=\operator{A}^*$.
When a function $g$ is a measurable real function, 
the maximal multiplication operator 
$\operator{M}[g]$, that is,
the multiplication operator with domain 
$\vectorspace{D}(\operator{M}[g])=\{\phi : \|g\,\phi\|<\infty\}$,
is self-adjoint.

We denote strong convergence as 
$\operator{A}_n\xrightarrow{s}\operator{A}$. Strong resolvent
convergence of a sequence of self-adjoint operators $\operator{A}_n$
to a self-adjoint operator $\operator{A}$, means that 
$(\operator{A}_n - z)^{-1}\xrightarrow{s}(\operator{A}-z)^{-1}$ 
for some nonreal $z$,
which
we denote as $\operator{A}_n\xrightarrow{srs}\operator{A}$
\cite{Kato:1995,Reed+Simon:1981,Segal+Kunze:1978,Weidmann:1980,Simon:2015}.
For a function
of a self-adjoint operator, we have the following useful theorem
about the convergence of functions of operators.
\begin{theorem}
  \label{thm:bounded-function-srs}
  If a sequence of self-adjoint operators $\operator{A}_n$ converges in strong
  resolvent sense to a self-adjoint operator $\operator{A}$, and $f$ is a bounded
  continuous function defined on $\mathbb{R}$, then $f(\operator{A}_n)$
  converges strongly to $f(\operator{A})$.
\end{theorem}
\begin{proof}
  See
  \cite[Theorem~VIII.20 (b)]{Reed+Simon:1981},
  \cite[Theorem~9.17]{Weidmann:1980}
  or \cite[Theorem~11.4]{Segal+Kunze:1978}.
\end{proof}

The spectral representation of a self-adjoint operator is
$%
  \operator{A} = \int_{\sigma(\operator{A})}t\,d\operator{E}(t),
$ %
where $\sigma(\operator{A})$ is the spectrum, and $\operator{E}(t)$ is the
spectral family of the operator $\operator{A}$
\cite{Stone:1932,Riesz+Nagy:1955,Akhiezer+Glazman:1993,Kato:1995,Halmos:1982,Kreyszig:1989,Reed+Simon:1981,Segal+Kunze:1978,Weidmann:1980,Schmudgen:2012,Simon:2015}.
For any function which is measurable
with respect to the spectral family, we have
\begin{equation}
  \label{eq:spectral-rep}
  f(\operator{A})=\int_{\sigma(\operator{A})}f(t)\,d\operator{E}(t).
\end{equation}
The spectral family can be defined as left or right
continuous for the parameter $t$ in strong convergence sense
\cite[Chapter~VI, Section~5.1]{Kato:1995}.
We adopt the usual convention of defining it as right continuous. Only points
of discontinuity of $\operator{E}(t)$ are the eigenvalues of $\operator{A}$
\cite[Theorem~7.23]{Weidmann:1980}. An \textit{eigenvalue}
of an operator is a value $\lambda\in\mathbb{C}$ that satisfies 
$%
  \lambda\,\phi=\operator{A}\,\phi.
$ %
  Function $\phi$ is called an \textit{eigenfunction} or an eigenvector.

An eigenvalue $\lambda$ of a multiplication operator $\operator{M}[g]$ has to
satisfy
\cite[p.~103, Example~1]{Weidmann:1980}
\begin{equation}
  \label{eq:eigenvalue}
  \int_S w(\mat{x})\,d\mat{x}>0
\end{equation}
for some set $S\subset\{\mat{x}\in\Omega : \lambda=g(\mat{x})\}$
and an eigenfunction has to be a scalar multiple of a characteristic function
$\chi_S(\mat{x})$.
The spectrum of the multiplication operator $\sigma(\operator{M}[g])$
for a real function $g$
is the essential range of the function $g$ 
\cite[Problem~67]{Halmos:1982}, 
\cite[Chapter~VIII.3, Proposition~1]{Reed+Simon:1981},
\cite[p.~103, Example~1]{Weidmann:1980},
that is, 
\begin{equation}
  \label{eq:essential-range}
  \sigma(\operator{M}[g])=
  \mathcal{R}(g)=
  \left\{
  y\in\mathbb{R} :
  \text{for all }\epsilon>0, 
  \int_{\{\mat{x}\in\Omega\,:\,\abs{g(\mat{x})-y}<\epsilon\}}
  w(\mat{x})\,d\mat{x}>0
  \right\}.
\end{equation}

For the convergence of the spectral families, we have the following theorem.
\begin{theorem}
  \label{thm:characteristic-convergence}
  Let $\operator{A}_n$ and $\operator{A}$ be self-adjoint operators with
  spectral families $\operator{E}_n(t)$ and $\operator{E}(t)$, respectively.
  If $\operator{A}_n\xrightarrow{srs}\operator{A}$, then we have the following
  equivalent results:
  \begin{enumerate}
    \item $\operator{E}_n(t)\xrightarrow{s}\operator{E}(t)$ when $t$ is
      not an eigenvalue of $\operator{A}$.
    \item For characteristic functions, we have
      $\chi_{[a,b]}(\operator{A}_n)\xrightarrow{s}\chi_{[a,b]}(\operator{A})$
      when $a$ and $b$ are not eigenvalues of $\operator{A}$.
  \end{enumerate}
\end{theorem}
\begin{proof}
  For 1, see \cite[Chapter~VIII, Theorem~1.15]{Kato:1995} or
  \cite[Theorem~9.19]{Weidmann:1980}. For 2, see
  \cite[Theorem~VIII.24 (b)]{Reed+Simon:1981},
  \cite[Theorem~11.4 e]{Segal+Kunze:1978}, or
  \cite[Chapter~7.2, problem~5]{Simon:2015}.
\end{proof}

A self-adjoint operator has an infinite matrix representation with
 matrix elements \eqref{eq:matrix-elements} if and only if
 the vectors
of the orthonormal
basis
$\phi_0,\phi_1,\ldots$ are dense in the underlying Hilbert space, and
\begin{equation}
  \label{eq:infinite-matrix-condition}
  \|\operator{A}\,\phi_i\|<\infty
\end{equation}
for all $i=0,1,\ldots$ (see \cite[Theorem~3.4]{Stone:1932} or 
\cite[Section~47]{Akhiezer+Glazman:1993}).
For a multiplication operator $\operator{M}[g]$, this
is equivalent to
\begin{equation}
  \label{eq:mult-mat}
  \int_{\Omega} 
  \abs{g(\mat{x})\,\phi_i(\mat{x})}^2\,w(\mat{x})\,d\mat{x}<\infty
\end{equation}
for all $i=0,1,\ldots$.

If $\operator{A}$ is a self-adjoint operator with an infinite matrix
representation $\mat{A}_\infty$, then also $\mat{A}_\infty$ is a self-adjoint
operator on $\ell^2$, the space of absolutely square-summable sequences. 
The domain of $\mat{A}_\infty$ is
$
\vectorspace{D}(\mat{A}_\infty)=
  \left\{
  \mat{v} : \|\mat{A}_\infty\,\mat{v}\|<\infty
  \right\}
$
which is isomorphic with $\vectorspace{D}(\operator{A})$.

Vectors $\mat{e}_i$ have $1$ in component $i$ and 0 in other components.
We use the same notation for $\mat{e}_i$ in $\mathbb{R}^n$ or $\ell^2$.
Vector space  $\ell^2_0$ is the space of finite linear combinations of
$\mat{e}_i\in\ell^2$. Matrix $\mat{I}_n$ is an identity matrix in
$\mathbb{R}^{n+1}$ and $\mat{I}_\infty$ in $\ell^2$. Infininte matrix of zeros
is $\mat{0}_{\infty}$, and $\mat{0}_{n\times m}$ are $n\times m$ matrices
of zeros where $n$ or $m$ can be $\infty$.

\subsection{Matrices with nonnegative elements}
\label{sec:matrices-with-nonnegative-elements}
We define an elementwise partial order of matrices as
  $\mat{A}\ewleq\mat{B}$,
  that is, $\mat{A}\ewleq\mat{B}$ if $[\mat{A}]_{i,j}\leq [\mat{B}]_{i,j}$ for all
  $i,j$.
For functions of finite matrices with nonnegative elements, we introduce
two classes of special interest \cite{Hansen:1992,Hansen:2000}.
\begin{definition}
  \label{def:m-positive}
  Let $f:I\mapsto\mathbb{R}$ be a real function defined on an interval
  $I\subset\mathbb{R}$.
  \begin{enumerate}
  \item
    $f$ is $m$-positive if $0\in I$ and $f(\mat{A})\ewgeq \mat{0}$ for every
    symmetric $n\times n$ matrix $\mat{A}\ewgeq \mat{0}$ with spectrum contained
    in $I$, and every $n\in \mathbb{N}$.
  \item
    $f$ is $m$-monotone if
    $\mat{A}\ewleq \mat{B}\Rightarrow f(\mat{A})\ewleq f(\mat{B})$,
    for all symmetric $n\times n$ matrices that satisfy 
    $\mat{A},\mat{B}\ewgeq \mat{0}$
      with spectra in $I$, and every $n\in\mathbb{N}$.
  \end{enumerate}  
\end{definition}
Under favorable conditions, the concepts of $m$-positive and $m$-monotone
functions are equal as is
shown by the following theorem.
\begin{theorem}
  \label{thm:m-positive-equal-m-monotone}
  Let $f:I\mapsto\mathbb{R}$ be a real function defined on an interval
  $I\subset\mathbb{R}$ such that $0\in I$ and $f(0)\geq 0$. Then
  $f$ is $m$-positive if and only if it is $m$-monotone.
\end{theorem}
\begin{proof}
  See \cite[Theorem~2.1~(1)]{Hansen:1992}.
\end{proof}
On $[0,\infty)$ and $(-\infty,\infty)$ intervals, 
all $m$-positive and hence $m$-monotone functions are absolutely 
monotone functions, that is, functions that have a converging
power series expansion $f(x)=\sum_{n=0}^\infty c_n\,x^n$ for
$n=0,1,2,\ldots$ and $c_n\geq 0$ for all $n$ 
\cite{Micchelli+Willoughby:1979,Hansen:1992,Hansen:2000}.

\subsection{Jensen's operator inequality}
\label{sec:jensens-operator-inequality}
For bounded operators, we say that a self-adjoint operator
$\operator{A}$ is positive
if
$
  0 \leq \langle \phi, \operator{A}\,\phi\rangle
$
for all
$\phi$ in the Hilbert space. We mark
this as $\operator{A}\qfgeq 0$.
For bounded operators, we can use definition $\operator{A}\qfleq \operator{B}$
if $0 \qfleq \operator{B}-\operator{A}$ 
\cite{Kato:1995,Reed+Simon:1981,Bhatia:1997,Carlen:2010}.
For unbounded self-adjoint operators, we define
$\operator{A}\qfleq \operator{B}$ if
$\langle \operatorname{sgn}(\operator{A})\,\sqrt{\abs{\operator{A}}}\,\phi,
\sqrt{\abs{\operator{A}}}\,\phi\rangle\leq
\langle \operatorname{sgn}(\operator{B})\,\sqrt{\abs{\operator{B}}}\,\phi,
\sqrt{\abs{\operator{B}}}\,\phi\rangle$ for all
$\phi\in\vectorspace{D}(\sqrt{\abs{\operator{B}}})\subset \vectorspace{D}(\sqrt{\abs{\operator{A}}})$ where
$\operatorname{sgn}(\cdot)$ is the sign function.

  A real function $f$ defined on an interval $I$ is said to be
  operator convex for bounded operators
  if for all self-adjoint bounded operators
  $\operator{A},\operator{B}$ %
  with spectrum in $I$,
  for each $\lambda\in[0,1]$, we have
 $   f(\lambda\,\operator{A}+(1-\lambda)\,\operator{B})\qfleq \lambda\,
    f(\operator{A})+(1-\lambda)\,f(\operator{B})
$ \cite{Bhatia:1997,Hansen+Pedersen:2003,Carlen:2010}.
Operators $f(\operator{A})$ and $f(\operator{B})$
are bounded because operator convex functions are continuous.

The set of operator convex functions does not contain all convex functions.
For example, for bounded operators on interval
$I=(0,\infty)$, $f(x)=x^p$ is operator convex for
  $-1\leq p\leq 0$ and $1\leq p \leq 2$ but not for 
  $p<-1$, $0 < p < 1$ or $2<p$ 
\cite[Theorem~2.6]{Carlen:2010}.

For the proof of convergence of the matrix approximation for some integrals,
we can use the following version of Jensen's operator inequality
\cite[Theorem~2.1 (i) and (iv)]{Hansen+Pedersen:2003}.
  For an operator convex function $f$ defined on an interval
  $I$
  \begin{equation}
    \label{eq:jensen}
    \operator{P}\,
    f(\operator{P}\,\operator{A}\,\operator{P}+s\,(1-\operator{P}))\,
    \operator{P}\qfleq \operator{P}\,f(\operator{A})\,\operator{P}
  \end{equation}
  for every orthogonal projection operator $\operator{P}$ 
  and every bounded
  self-adjoint operator $\operator{A}$ defined 
  on an infinite-dimensional Hilbert space
  with spectrum
  in $I$ and every $s\in I$.

Another relevant concept is operator monotone functions 
\cite{Bhatia:1997,Carlen:2010}.
A real function $f$ defined on interval $I$ is operator monotone, if for
all self-adjoint operators $\operator{A},\operator{B}$ with sepectrum on
$I$, we have
$
  \operator{A}\qfleq \operator{B} 
  \Rightarrow f(\operator{A})\qfleq f(\operator{B}).
$
For example, $f(x)=-x^p$ is operator monotone for $-1\leq p \leq 0$ on interval
$x\in(0,\infty)$ \cite[Theorem~2.6]{Carlen:2010}. The definition of operator
monotone functions is the same for unbounded operators. Also, functions
that are operator monotone for bounded operators on interval $(0,\infty)$
are operator monotone
for unbounded positive self-adjoint operators \cite[Theorem~5]{Dinh+Tikhonov:2010}.

\section{Convergence}
\label{sec:convergence}
We establish convergence for a growing class of outside functions $f$. We start
by first proving the strong resolvent convergence that then immediately
covers the
bounded continuous outside functions in
Section~\ref{sec:proof-of-strong-resolvent-convergence}.
Our first extension to this basic result is the class of discontinuous
functions in Section~\ref{sec:convergence-for-discontinuous-functions}.

Then we extend the results to unbounded functions.
First, we discuss the topic of general unbounded outside functions in
Section~\ref{sec:strong-resolvent-convergence-for-unbounded-functions}
and notice that it is a too wide class of functions to prove convergence
in general. Then we find proof for a growing class of unbounded functions
as the inside function $g$ and basis functions $\phi_i$ are polynomials.
In Section~\ref{sec:convergence-for-quadratically-bounded-functions},
we prove convergence for quadratically bounded outside functions
without restrictions to the inside function and polynomially bounded
outside functions when the inside function and the basis functions are 
polynomials.
In Section~\ref{sec:nonnegative-matrix-coefficients},
we prove convergence for inside functions that have matrix
representations with nonnegative coefficients. The outside functions do not
even have to meet the criteria of being functions of operators that have
infinite matrix representation.

Finally, we consider convergence for outside functions that are singular on a
finite endpoint of an integration interval in
Section~\ref{sec:improper-integrals-on-finite-endpoints-of-an-interval}.

\subsection{Proof of the strong resolvent convergence}
\label{sec:proof-of-strong-resolvent-convergence}
It can be difficult to prove directly that a sequence of self-adjoint
operators converges in the strong resolvent sense. Therefore, we will
introduce a concept of core that we can use to give proof of strong
resolvent convergence.
\begin{definition}
A core $\vectorspace{D}_0$ of
an operator $\operator{A}$ is a subspace of the domain of the operator
such that
$
(\operator{A}\vert_{\vectorspace{D}_0})^{**}=\operator{A}.
$
\end{definition}
Usually, the core is defined in terms of the closure of the operator
\cite{Kato:1995,Reed+Simon:1981,Weidmann:1980,Simon:2015} but
we note that for densely defined and closable operators,
the closure is equivalent to the second adjoint
\cite[Chapter~III, Theorem~5.29]{Kato:1995}, 
\cite[Theorem~5.3~(b)]{Weidmann:1980},
\cite[Theorem~VIII.1~(b)]{Reed+Simon:1981}
or 
      \cite[Theorem~7.1.1~(c)]{Simon:2015}.
With a suitable core, the strong resolvent convergence can be proved by
the following theorem.
\begin{theorem}
  Let $\vectorspace{D}_0$ be a core of self-adjoint operators $\operator{A}_n$
  and $\operator{A}$ where $n = 0,1,,\ldots$. If
  for each $\phi\in\vectorspace{D}_0$,
  \begin{equation*}
  \lim_{n\rightarrow \infty}\|(\operator{A}-\operator{A}_n)\,\phi\|=0,
  \end{equation*}
  then $\operator{A}_n\xrightarrow{srs}\operator{A}$. 
\end{theorem}
\begin{proof}
  See   \cite[Theorem~VIII.25~(a)]{Reed+Simon:1981} for proof, or
  \cite[Chapter~VIII, Corollary~1.6]{Kato:1995},
  \cite[Theorem~9.16~(i)]{Weidmann:1980}, or 
  \cite[Theorem~7.2.11]{Simon:2015} for proofs of similar but slightly more
  general theorems.
\end{proof}
After finding a suitable core, we can prove the strong resolvent convergence. The space of absolutely square-summable sequences or infinite-dimensional vectors is $\ell^2$. Its subspace $\ell^2_0$ is a space where vectors have only finitely many non-zero elements. It now turns out that $\ell^2_0$ is, in general, a core for any self-adjoint infinite matrix. 
Similar ideas related to infinite band matrices have been discussed in 
\cite{Barrios+Lopez+Martinez-Finkelshtein+Torrano:1999,Dombrowski:1990}.
A space which is isomorphic with $\ell^2_0$ has been used as a core
without a proof that it is a core in
\cite{Arai:1992}.
A space that contains a space that is isomorphic with $\ell^2_0$ has
also been used as a core in \cite{Hansen:2008}.
A proof that a vector space is a core for a certain operator
can sometimes be very specialized 
\cite[Lemma~4.4]{Rosler:2019}, but
we present a very general proof that $\ell^2_0$ is a core for self-adjoint
infinite matrices
or equivalently a space that is isomorphic with $\ell^2_0$ is a core
for a self-adjoint operator with infinite matrix representation.
\begin{theorem}
  \label{thm:core-srs}
  Let $\operator{A}$ be a self-adjoint operator with an infinite
  matrix representation $\mat{A}_\infty$.
  Then $\ell^2_0$ is a core for $\mat{A}_\infty$.
\end{theorem}
\begin{proof}
  Our proof is an adaptation of proof of \cite[Theorem~6.20]{Weidmann:1980}.
  We use the following property of the adjoint:
  \begin{itemize}
    \item Let $\operator{A}$ and $\operator{B}$ be densely defined operators,
      then 
      $\vectorspace{D}(\operator{A})\subset\vectorspace{D}(\operator{B})$
      and $\operator{A}=\operator{B}\vert_{\vectorspace{D}(\operator{A})}$ 
      $\Rightarrow$ 
      $\vectorspace{D}(\operator{B}^*)\subset\vectorspace{D}(\operator{A}^*)$
      and $\operator{B}^*=\operator{A}^*\vert_{\vectorspace{D}(\operator{B}^*)}$,
      that is,
      \begin{equation}
        \label{eq:operator-subset}
        \operator{A}\subset\operator{B}
        \Rightarrow
        \operator{B}^*\subset\operator{A}^*.
      \end{equation}
      (See \cite[p.~72]{Weidmann:1980} or \cite[p.~252]{Reed+Simon:1981}.)
  \end{itemize}
  Let operator $\operator{A}_0=\mat{A}_\infty\vert_{\ell^2_0}$.
  The natural basis vectors of $\ell^2$ are vectors $\mat{e}_i$.
  Vectors of $\ell^2_0$ are finite linear combinations of basis vectors
  $\mat{e}_i$.
  Because all basis vectors $\mat{e}_i\in\vectorspace{D}(\mat{A}_\infty)$,
  we also have $\ell^2_0\subset\vectorspace{D}(\mat{A}_\infty)$.
  Therefore, we have $\operator{A}_0\subset\mat{A}_\infty$.
  By \eqref{eq:operator-subset}, we have 
  $\mat{A}_\infty^*\subset \operator{A}_0^*$. Because $\mat{A}_\infty$
  is self-adjoint, we have 
  $\mat{A}_\infty=\mat{A}_\infty^*\subset\operator{A}_0^*$.
  We show that also
  $\operator{A}_0^*\subset\mat{A}_\infty$, that is,
  $
    \vectorspace{D}(\operator{A}_0^*)
    \subset
    \vectorspace{D}(\mat{A}_\infty)
    =
    \{\mat{u} : \|\mat{A}_\infty\,\mat{u}\|^2<\infty\}
    $
  and
  $\operator{A}_0^*=\mat{A}_\infty\vert_{\vectorspace{D}(\operator{A}_0^*)}$. 
  Let $\mat{u}\in\vectorspace{D}(\operator{A}_0^*)$ and thus
  $\|\operator{A}_0^*\,\mat{u}\|^2<\infty$. 
  Because $\mat{e}_k\in\vectorspace{D}(\operator{A}_0)$, we have for all
  $k=0,1,\ldots$
  \begin{equation}
    \label{eq:el-square-zero-dual}
    \langle \mat{e}_k,\operator{A}_0^*\,\mat{u}\rangle
    =
    \langle \operator{A}_0\,\mat{e}_k,\mat{u}\rangle
    =\sum_{j=0}^\infty \overline{a_{j,k}}\,u_j
    =\sum_{j=0}^\infty a_{k,j}\,u_j
    =\langle \mat{e}_k,\mat{A}_\infty\,\mat{u} \rangle.
  \end{equation}
  We square and sum over $k$
  \begin{equation*}
    \|\mat{A}_\infty\,\mat{u}\|^2
    =
    \sum_{k=0}^\infty\abs{\langle \mat{e}_k,\mat{A}_\infty\,\mat{u}\rangle}^2
    =
    \sum_{k=0}^\infty\abs{\langle \mat{e}_k,\operator{A}_0^*\,\mat{u}\rangle}^2
    =
    \|\operator{A}_0^*\,\mat{u}\|^2<\infty,
  \end{equation*}
  that is,
  $\mat{u}\in\vectorspace{D}(\operator{A}_0^*)\Rightarrow
  \mat{u}\in\vectorspace{D}(\mat{A}_\infty)$.
  By \eqref{eq:el-square-zero-dual} we also have for any 
  $\mat{u}\in\vectorspace{D}(\operator{A}_0^*)$
  \begin{equation*}
    \operator{A}_0^*\,\mat{u}=
    \sum_{k=0}^\infty
    \langle \mat{e}_k,\operator{A}_0^*\,\mat{u}\rangle\,
    \mat{e}_k=
    \sum_{k=0}^\infty
    \langle \mat{e}_k,\mat{A}_\infty\,\mat{u}\rangle\,
    \mat{e}_k=\mat{A}_\infty\,\mat{u},
  \end{equation*}
  that is, 
  $\operator{A}_0^*=\mat{A}_\infty\vert_{\vectorspace{D}(\operator{A}_0^*)}$.
  Because we have $\mat{A}_\infty\subset\operator{A}_0^*$ and 
  $\operator{A}_0^*\subset\mat{A}_\infty$, we have 
  $\operator{A}_0^*=\mat{A}_\infty$. Therefore, we also have
  $\operator{A}_0^{**}=\mat{A}_\infty^*=\mat{A}_\infty$.
\end{proof}

After discovering the core, it is straightforward to prove the 
strong resolvent convergence.
\begin{theorem}
  \label{thm:srs-convergence}
  Let $\operator{A}$ be a self-adjoint operator with an infinite matrix
  representation
$\mat{A}_\infty$. Let finite matrices $\mat{A}_n$ be the leading
principal $(n+1)\times (n+1)$ 
submatrices of $\mat{A}_\infty$. Then
for any $\mat{v}\in\ell^2_0$,
\begin{equation}
  \label{eq:srs}
  \mat{\tilde{A}}_n=
  \left[
    \begin{array}{cc}
    \mat{A}_n                & \mat{0}_{(n+1)\times\infty} \\
    \mat{0}_{\infty\times (n+1)} & \mat{0}_{\infty} 
    \end{array}
    \right]
  \xrightarrow{srs} \mat{A}_\infty,
\end{equation}
because $\mat{\tilde{A}}_n\,\mat{v}\rightarrow\mat{A}_\infty\,\mat{v}$
as $n\rightarrow\infty$.
\end{theorem}
\begin{proof}
  We define an orthonormal projection
  \begin{equation}
    \label{eq:projection}
    \mat{P}_n=    \left[
      \begin{array}{cc}
        \mat{I}_n                & \mat{0}_{(n+1)\times\infty} \\
        \mat{0}_{\infty\times (n+1)} & \mat{0}_\infty 
      \end{array}
      \right].
  \end{equation}
  With this projection, we can write
  $\mat{\tilde{A}}_n=\mat{P}_n\,\mat{A}_\infty\,\mat{P}_n$.
  We take an arbitrary vector in the core $\mat{v}\in\ell^2_0$. Let $m$ be the
  last index where components of $\mat{v}$ are non-zero, that is,
  $v_i=0$ for all $i>m$. For all $n>m$, we have
  $\mat{P}_n\,\mat{v}=\mat{v}$, and thus
  $(\mat{\tilde{A}}_n-\mat{A}_\infty)\,\mat{v}=(\mat{P}_n-\mat{I}_\infty)\,\mat{A}_\infty\,\mat{v}$.
  Vector $\mat{u}=\mat{A}_\infty\,\mat{v}\in\ell^2$, and therefore we have
  $(\mat{P}_n-\mat{I}_\infty)\,\mat{u}\rightarrow\mat{0}$ as
  $n\rightarrow\infty$. The strong resolvent convergence in \eqref{eq:srs}
  follows from Theorem~\ref{thm:core-srs}.
\end{proof}

From this, we already have convergence for bounded and continuous functions.
\begin{theorem}
  Let orthonormal functions $\phi_0=1,\phi_1,\phi_2,\ldots$ be dense in
  $\vectorspace{L}^2_w(\Omega)$. Also, let function
  $g$ satisfy \eqref{eq:mult-mat} and matrices $\mat{M}_n[g]$ have
  elements as in \eqref{eq:matrix-elements}.
  Then for bounded and continuous $f$
\begin{equation}
  \label{eq:bounded-convergence}
  \left[
    \begin{array}{cc}
      f\left(
      \mat{M}_n[g]
      \right)
                              & \mat{0}_{(n+1)\times\infty} \\
      \mat{0}_{\infty\times (n+1)} & \mat{0}_\infty
    \end{array}
    \right]
  =
  f\left(
  \left[
    \begin{array}{cc}
      \mat{M}_n[g]             & \mat{0}_{(n+1)\times\infty} \\
      \mat{0}_{\infty\times (n+1)} & \mat{0}_\infty 
    \end{array}
    \right]
  \right)
  \xrightarrow{s}
  f(\mat{M}_\infty[g]),
\end{equation}
and, as $n\rightarrow\infty$,
\begin{equation*}
\mat{e}_0^\top\,f(\mat{M}_n[g])\,\mat{e}_0
\rightarrow
\mat{e}_0^\top\,f(\mat{M}_\infty[g])\,\mat{e}_0
=
\langle 1,f(\operator{M}[g]\,1\rangle
=
\int_\Omega f(g(\mat{x}))\,w(\mat{x})\,d\mat{x}.
\end{equation*}
\end{theorem}
\begin{proof}
  By Theorem~\ref{thm:srs-convergence}, we have
  $\mat{P}_n\,\mat{M}_\infty[g]\,\mat{P}_n\xrightarrow{srs}\mat{M}_\infty[g]$
  where $\mat{P}_n$ is as in \eqref{eq:projection}.
  The convergence of \eqref{eq:bounded-convergence} follows from
  Theorem~\ref{thm:bounded-function-srs}.
  The convergence of the quadratic form
  $\mat{e}_o^\top\,f(\mat{M}_n[g])\,\mat{e}_0\rightarrow
  \mat{e}_0^\top\,f(\mat{M}_\infty[g])\,\mat{e}_0$ is a special case of
  the weak convergence that follows from the strong convergence
  \eqref{eq:bounded-convergence}.
\end{proof}

This result already extends the previous result
\cite[Theorem~3]{Sarmavuori+Sarkka:2019}
for inside functions
from bounded functions $g$ to unbounded functions $g$.
However, our aim is to further extend the results to discontinuous and unbounded outside functions $f$ as well.

\subsection{Convergence for discontinuous functions}
\label{sec:convergence-for-discontinuous-functions}
First, we consider convergence when the integrand function $f$ in 
\eqref{eq:numint-def} is bounded but discontinuous. The largest space of 
discontinuous functions that are possible to integrate numerically are the
Riemann--Stieltjes integrable functions. 
It does not seem to be
possible to obtain convergence results for a larger set of functions
like, for example,
Lebesgue--Stieltjes integrable functions 
\cite[Chapter~1.8]{Davis+Rabinowitz:1984}. Similar arguments rule out
convergence for Borel or Baire measurable functions that are not
Riemann--Stieltjes
integrable.
When the starting point is a Borel measure, the largest set of
discontinuous functions that the convergence theorem covers seem to be
 functions
that are discontinuous on a set that is contained in a closed set which has
zero measure with respect to the spectral family of the operator
\cite[Theorem~2.6]{Bade:1954}, \cite{Foguel:1958}, 
\cite[Theorem~3]{Simpson:1966}
or
\cite[Theorem~3]{Sarmavuori+Sarkka:2019}. 
These convergence theorems do not then
cover, for example,
Thomae's function \cite[Example~7.1.6]{Bartle+Sherbert:2011}
which is discontinuous on rational numbers but is
Riemann--Stieltjes integrable.
The definition of the Riemann-- and the related Darboux--Stieltjes integrals
are
 given in Appendix~\ref{app:integrable-functions}. 

The spectral integral is usually defined as a Borel or a Baire measure.
One example of spectral integral as a Riemann--Stieltjes integral
can be found in \cite[Chapter~4.1]{Schmudgen:2012}. However,
the Riemann--Stieltjes integral is defined there in norm operator topology,
and we will rather make the definition in strong operator topology because 
the finite matrix approximations
can converge in norm only if the operator is compact
\cite[Problem~175]{Halmos:1982}, 
\cite[Theorem~1.2]{Morrison:1995}
and
the function space defined in the norm operator topology can only be a subset
of the space defined in the strong operator topology.
We define the spectral integral as Darboux--Stieltjes integral
in the same manner as it is defined as Lebesgue--Stieltjes integral
in \cite[Section~126]{Riesz+Nagy:1955}.
The Riemann--Stieltjes integral is then the same as the Darboux--Stieltjes
integral except limited to functions that are continuous at the discontinuities
of the spectral family.
We can define a non-decreasing function $\rho(t)$ through the spectral family
as $\rho(t)=\|\operator{E}(t)\,\phi\|^2$ for any $\phi$ in the Hilbert space.
Because operators $\operator{E}(t)$ are projections,
$\|\operator{E}(t)\,\phi\|^2=\langle \phi,\operator{E}(t)\,\phi\rangle$. 
Thus, we can define
\begin{equation*}
  \langle \phi, f(\operator{A})\,\phi\rangle=
  \int_{\sigma(\operator{A})}f(t)\,d\langle \phi,\operator{E}(t)\,\phi\rangle
  =
  \int_{\sigma(\operator{A})}f(t)\,d\|\operator{E}(t)\,\phi\|^2
\end{equation*}
as a Darboux--Stieltjes integral. We say that a function $f$ is
Darboux--Stieltjes
integrable with respect to $\|\operator{E}(t)\,\phi\|^2$ if the
Darboux--Stieltjes integral exists.

We can use polarization identity \cite[Section~126]{Riesz+Nagy:1955}
\begin{equation}
  \label{eq:polarization}
  \langle \phi, \operator{E}(t)\,\psi\rangle
  =
  \left\|
  \operator{E}(t)\,\frac{\scriptstyle\phi+\psi}{\scriptstyle 2}
  \right\|^2
  -
  \left\|
  \operator{E}(t)\,\frac{\scriptstyle\phi-\psi}{\scriptstyle 2}
  \right\|^2
  +i\,
  \left\|
  \operator{E}(t)\,\frac{\scriptstyle \phi+i\,\psi}{\scriptstyle 2}
  \right\|^2
  -i\,
  \left\|
  \operator{E}(t)\,\frac{\scriptstyle\phi-i\,\psi}{\scriptstyle 2}
  \right\|^2
\end{equation}
to define Darboux--Stieltjes integral with respect to
$\rho(t)=\langle \phi,\operator{E}(t)\,\psi\rangle$ as a linear
combination of four Darboux--Stieltjes integrals. Finally, if $f$ is
Darboux--Stieltjes integrable with respect to $\|\operator{E}(t)\,\phi\|^2$
for all $\phi$ in the Hilbert space, then we say that $f$ is Darboux--Stieltjes
integrable with respect to $\operator{E}(t)$.

We aim to extend the convergence of Theorem~\ref{thm:bounded-function-srs}
for continuous functions to Darboux-- or Riemann--Stieltjes
integrable functions.
We start by proving the equivalent of Theorem~\ref{thm:bounded-function-srs}
for step functions.
\begin{lemma}
  \label{lemma:step-function-conv}
  Let $\operator{A}_n$ and $\operator{A}$ be self-adjoint operators with
  spectral families $\operator{E}_n(t)$ and $\operator{E}(t)$, respectively, for $n=0,1,2,\ldots$.
  If $\operator{A}_n\xrightarrow{srs}\operator{A}$, then
  for a step function 
  $f(\operator{A}_n)\xrightarrow{s}f(\operator{A})$ provided that
  discontinuity points of $f$ are not eigenvalues of $\operator{A}$.
\end{lemma}
\begin{proof}
  A step function can be represented as a finite linear combination of
  characteristic functions. By Theorem~\ref{thm:characteristic-convergence},
  the convergence is strong for each characteristic function. A finite
  linear combination of operators that converge strongly also converges
  strongly \cite[Chapter~4.9, Problem~2]{Kreyszig:1989}.
\end{proof}
 
Next, we extend convergence from step functions to 
Darboux--Stieltjes integrable functions.
\begin{theorem}
  \label{thm:Darboux-convergence}
  Let $\operator{A}_n$ and $\operator{A}$ be self-adjoint operators with
  spectral families $\operator{E}_n(t)$ and $\operator{E}(t)$, respectively, for $n=0,1,2,\ldots$.
  Let function $f$ be continuous on eigenvalues of $\operator{A}$.
  Let $\operator{A}_n\xrightarrow{srs}\operator{A}$.
  Then the following statements hold:
  \begin{enumerate}
  \item If $f(t)$ is Darboux--Stieltjes integrable with respect to
    $\rho(t)=\|\operator{E}(t)\,\phi\|^2$, then
    \begin{equation*}
      \lim_{n\rightarrow \infty}\langle \phi, f(\operator{A}_n)\,\phi\rangle
      =\langle \phi, f(\operator{A})\,\phi\rangle.
    \end{equation*}
  \item If $f(t)$ is Darboux--Stieltjes integrable with respect to
    $\rho(t)=\langle \psi,\operator{E}(t)\,\phi\rangle$, then
    \begin{equation*}
      \lim_{n\rightarrow \infty}\langle \psi, f(\operator{A}_n)\,\phi\rangle
      =\langle \psi, f(\operator{A})\,\phi\rangle.
    \end{equation*}
  \item If $f(t)$ is Darboux--Stieltjes integrable with respect to
    $\operator{E}(t)$, then
    $
      f(\operator{A}_n)\xrightarrow{s}f(\operator{A}).
$
  \end{enumerate}
\end{theorem}
\begin{proof}
  For the proof of statement 1, 
    we take the idea from \cite[Chapter~2.7.8]{Davis+Rabinowitz:1984}.
    We assume that $f$ is Darboux--Stieltjes integrable with respect to
    $\rho(t)=\|\operator{E}(t)\,\phi\|^2$.
    We define step functions $f_l$ and $f_u$ as in
    Lemma~\ref{lemma:Darboux--Stieltjes}. By Lemma~\ref{lemma:Darboux--Stieltjes},
    we can also select $f_l$ and $f_u$ so that they are continuous at
    eigenvalues of $\operator{A}$. By Lemma~\ref{lemma:step-function-conv},
    we have
 $     f_l(\operator{A}_n)\xrightarrow{s}f_l(\operator{A})$ and
 $     f_u(\operator{A}_n)\xrightarrow{s}f_u(\operator{A})$.
    Because strong convergence implies weak convergence, we have also
    \begin{align*}
      \langle \phi,f_l(\operator{A}_n)\,\phi\rangle
      &\rightarrow
      \langle \phi, f_l(\operator{A})\,\phi\rangle,\\
      \langle \phi,f_u(\operator{A}_n)\,\phi\rangle
      &\rightarrow
      \langle \phi, f_u(\operator{A})\,\phi\rangle.
    \end{align*}
    We see from the definition of $f_l$ and $f_u$ in
    Lemma~\ref{lemma:Darboux--Stieltjes} that
    because, for example,
    \begin{align*}
      \langle \phi, f_l(\operator{A})\,\phi\rangle
      &=
      \int_{\sigma(\operator{A})}f_l(t)\,d\|\operator{E}(t)\,\phi\|^2
      \leq
      \int_{\sigma(\operator{A})}f(t)\,d\|\operator{E}(t)\,\phi\|^2
      =
      \langle \phi, f(\operator{A})\,\phi\rangle,
    \end{align*}
    we have
    \begin{equation*}
      \langle \phi, f_l(\operator{A})\,\phi\rangle\leq
      \langle \phi, f(\operator{A})\,\phi\rangle\leq
      \langle \phi, f_u(\operator{A})\,\phi\rangle,
    \end{equation*}
    and
    \begin{equation*}
      \langle \phi, f_l(\operator{A}_n)\,\phi\rangle\leq
      \langle \phi, f(\operator{A}_n)\,\phi\rangle\leq
      \langle \phi, f_u(\operator{A}_n)\,\phi\rangle
    \end{equation*}
    as well.
    By Lemma~\ref{lemma:Darboux--Stieltjes} we also have
  $\langle \phi, (f_u(\operator{A})-f_l(\operator{A}))\,\phi\rangle<\epsilon$,
    and thus, we get the following inequalities:
    \begin{align*}
      \langle \phi,f(\operator{A})\,\phi\rangle-\epsilon
      &<
      \langle\phi, f_l(\operator{A})\,\phi\rangle
      =
      \liminf_{n\rightarrow\infty} \langle\phi, f_l(\operator{A}_n)\,\phi\rangle,\\
      \langle \phi,f(\operator{A})\,\phi\rangle+\epsilon
      &>
      \langle\phi, f_u(\operator{A})\,\phi\rangle
      =
      \limsup_{n\rightarrow\infty} \langle\phi, f_u(\operator{A}_n)\,\phi\rangle.
    \end{align*}
    By \cite[Theorem~3.17 (b)]{Rudin:1976}, there is N so that for all
    $n\geq N$,
    \begin{equation*}    
      \langle \phi,f(\operator{A})\,\phi\rangle-\epsilon
      <
      \langle \phi, f_l(\operator{A}_n)\,\phi\rangle
      \leq
      \langle \phi, f(\operator{A}_n)\,\phi\rangle
      \leq
      \langle \phi, f_u(\operator{A}_n)\,\phi\rangle
      <
      \langle\phi, f(\operator{A})\,\phi\rangle+\epsilon,
    \end{equation*}
    and hence
   $ \langle \phi, f(\operator{A}_n)\,\phi\rangle
    \rightarrow\langle \phi, f(\operator{A})\,\phi\rangle$ as
    $n\rightarrow\infty$.

    The proof of statement 2 follows from 1 by the polarization identity
    \eqref{eq:polarization}.
    For the proof of statement 3, we take an arbitrary $\phi$, and set
    $\psi=f(\operator{A})\,\phi$. Then
    \begin{align*}
      \|(f(\operator{A}_n)-f(\operator{A}))\,\phi\|
      & =
      \langle \phi, f(\operator{A}_n)^2\,\phi\rangle-
      \langle \psi, f(\operator{A}_n)\,\phi\rangle-
      \langle f(\operator{A}_n)\,\phi,\psi\rangle+
      \langle \psi, \psi\rangle\\
      & \rightarrow
      \langle \phi, f(\operator{A})^2\,\phi\rangle-
      \langle \psi, f(\operator{A})\,\phi\rangle-
      \langle f(\operator{A})\,\phi,\psi\rangle+
      \langle \phi, f(\operator{A})^2\,\phi\rangle\\      
      &= 0
    \end{align*}
    as $n\rightarrow\infty$ because $f^2$ is Darboux--Stieltjes integrable
    when $f$ is \cite[Theorem~6.13]{Rudin:1976}.
\end{proof}
By this theorem, we can extend the convergence of
\eqref{eq:bounded-convergence} to functions $f$ that are discontinuous
but Darboux--Stieltjes integrable with respect to $\operator{E}(t)$.
As noted in \cite[Theorem~4]{Sarmavuori+Sarkka:2019},
because we have strong convergence, we also have convergence for numerical
integrals of products of functions \eqref{eq:product-convergence} when
all the functions are Darboux--Stieltjes integrable.

It is good to note that the eigenvalues of $\operator{A}$
are the only points of discontinuity of the family
$\operator{E}(t)$, and points of continuity $t_0$ of
$\operator{E}(t)$ are also the points of continuity for 
$\rho(t)=\langle \phi,\operator{E}(t)\,\psi\rangle$.
The reason for this is because, by definition, 
$\operator{E}(t)\xrightarrow{s}\operator{E}(t_0)$ when
$t\rightarrow t_0$. Strong convergence implies weak convergence
\cite{Riesz+Nagy:1955,Akhiezer+Glazman:1993,Kato:1995,Kreyszig:1989,Reed+Simon:1981,Weidmann:1980,Simon:2015}, which means that 
$\langle\phi,\operator{E}(t)\,\psi\rangle\rightarrow 
\langle\phi,\operator{E}(t_0)\,\psi\rangle$ for all $\phi,\psi$ in
the Hilbert space. It is still possible that for some
$\phi$ and $\psi$, 
$\langle \phi, \operator{E}(t)\,\psi\rangle$ is continuous
at some points that are eigenvalues of $\operator{A}$ and hence
where $\operator{E}(t)$ is not continuous.

If a function $f$
is Riemann--Stieltjes integrable with respect to
$\rho(t)=\|\operator{E}(t)\,\phi\|$ for all $\phi$, then $f$ must be continuous
on all eigenvalues of $\operator{A}$ \cite[p.~251]{Kestelman:1960}.
Then we can say that
$f$ is Riemann--Stieltjes integrable with respect to $\operator{E}(t)$
or with respect to $\operator{A}$.
By Theorem~\ref{thm:Darboux-convergence} statement 3,
we always have convergence for Riemann--Stieltjes integrable functions.
Therefore, we notice that Riemann--Stieltjes integrability is a suitable
criterion for convergence in terms of discontinuity of the function.
In order to have a discontinuity criterion for unbounded functions alike,
we extend the definition of Riemann--Stieltjes integrability in
the following manner.
\begin{definition}
  Function $f$ is Riemann--Stieltjes integrable with respect to
  self-adjoint operator $\operator{A}$ if $f$ is Riemann--Stieltjes
  integrable with respect to the spectral family of $\operator{A}$
  over any finite interval with endpoints that are not eigenvalues
  of $\operator{A}$.
\end{definition}
If an unbounded function $f$ is Riemann--Stieltjes integrable with respect
to self-adjoint operator $\operator{A}$,
then for all $\phi\in\vectorspace{D}(\operator{A})$, there exists
$f(\operator{A})\,\phi=\int_{-\infty}^\infty f(t)\,d\operator{E}(t)\,\phi$
as an improper integral that is a limit of Riemann--Stieltjes integrals
over a finite interval.

\subsection{Strong resolvent convergence for unbounded functions}
\label{sec:strong-resolvent-convergence-for-unbounded-functions}
We can now show that an unbounded function of an operator converges in the
strong resolvent sense. However, as discussed below the result of this theorem
looks more useful than it actually is.

\begin{theorem}
  \label{thm:unbounded-function-srs}
  Let $\operator{A}_n$ and $\operator{A}$ be self-adjoint operators, and let
  $f$ be such a continuous real function defined on $\mathbb{R}$
  that $f(\operator{A}_n)$ and $f(\operator{A})$ are densely defined. 
  Then
  \begin{equation*}
    \operator{A}_n\xrightarrow{srs}\operator{A}
    \Rightarrow
    f(\operator{A}_n)\xrightarrow{srs}f(\operator{A}).
  \end{equation*}
\end{theorem}
\begin{proof}
  Operators $\operator{B}_n=f(\operator{A}_n)$ and
  $\operator{B}=f(\operator{A})$ are self-adjoint.
  Function
  $
  \tilde{f}(x)=(f(x)-z)^{-1}
  $
  is bounded and continuous for nonreal $z$.
  By Theorem~\ref{thm:bounded-function-srs},
  $\tilde{f}(\operator{A}_n)\xrightarrow{s}\tilde{f}(\operator{A})$, and
  we have
  $
    (\operator{B}_n-z)^{-1}=\tilde{f}(\operator{A}_n)
    \xrightarrow{s}\tilde{f}(\operator{A})=(\operator{B}-z)^{-1},
$
  which means
  by definition of the strong resolvent convergence that
  $f(\operator{A}_n)=\operator{B}_n\xrightarrow{srs}\operator{B}=
  f(\operator{A})$.
\end{proof}
It is now important to note that although the above theorem ensures 
strong resolvent convergence, it does not guarantee convergence of the inner product
$\langle \phi, f(\operator{A}_n)\,\phi\rangle$, which is what we want.
It only guarantees the existence of a sequence $\phi_n\rightarrow \phi$, for which
$\langle \phi, f(\operator{A}_n)\,\phi_n\rangle$ converges, as is shown
by the following theorem.
\begin{theorem}
  \label{thm:graph-convergence}
  Let self-adjoint operators $\operator{A}_n$ and $\operator{A}$ be
  such that $\operator{A}_n\xrightarrow{srs}\operator{A}$.
  Then
  for each $\phi\in\vectorspace{D}(\operator{A})$, there is a sequence
  $\phi_n\in\vectorspace{D}(\operator{A}_n)$ such that
  $
  \|\phi_n-\phi\|+\|\operator{A}_n\,\phi_n-\operator{A}\,\phi\|\rightarrow 0
  $
    as $n\rightarrow \infty.
    $
\end{theorem}
\begin{proof}
  See \cite[Satz (14)]{Stummel:1972}.
\end{proof}

An example sequence of $\phi_n$ is
$\phi_n=(\operator{A}_n-z)^{-1}\,(\operator{A}-z)\,\phi$, where $z$ is any
nonreal number. For nonreal $z$, the sequence $\phi_n$ is not real-valued.
However, it is often possible to find a real vector sequence $(\phi_n)$
based on a dominating 
function $h^2\geq \abs{f}$ such that 
$\langle\phi_n,f(\operator{A}_n)\,\phi_n\rangle\rightarrow\langle\phi,\operator{A}\,\phi\rangle$. 
In the case of numerical integration, this changes the construction of the
weights of the quadrature rule so that they are based on the function
sequence
$(\phi_n)$ instead of a fixed function $\phi$, but the convergence
of such quadrature can be proved as is shown by the following theorem.
\begin{theorem}
  \label{thm:reweighted-convergence}
  Let self-adjoint operators $\operator{A}_n$ and $\operator{A}$ be
  such that $\operator{A}_n\xrightarrow{srs}\operator{A}$.
  Let $f$ and $h$ be Riemann--Stieltjes integrable functions with respect to
  $\operator{A}$ so that
  $\abs{f(\mat{x})}\leq h^2(\mat{x})$,
  and $h^2(\mat{x})\geq c > 0$ for all
  $x\in\sigma(\operator{A})\cap\sigma(\operator{A}_n)$.
    For $\phi\in\vectorspace{D}(h(\operator{A}))$, as $n\rightarrow\infty$,
  \begin{equation*}
    (h(\operator{A}_n))^{-1}\,f(\operator{A}_n)\,
    (h(\operator{A}_n))^{-1}\,h(\operator{A})\,\phi
    \rightarrow (h(\operator{A}))^{-1}\,f(\operator{A})\,\phi
  \end{equation*}
  and hence
  \begin{equation*}
    \langle h(\operator{A})\,\phi, (h(\operator{A}_n))^{-1}\,f(\operator{A}_n)\,
    (h(\operator{A}_n))^{-1}\,h(\operator{A})\,\phi\rangle
    \rightarrow \langle \phi,f(\operator{A})\,\phi\rangle.
  \end{equation*}
\end{theorem}
\begin{proof}
  Function $q(\mat{x})=\frac{f(\mat{x})}{h^2(\mat{x})}$ is bounded so that
  $\abs{q(\mat{x})}\leq 1$. Function $q$ is Riemann--Stieltjes integrable
  \cite[Theorems 6.11 and 6.13]{Rudin:1976}.
  Therefore, the operators
  $\operator{B}_n=(h(\operator{A}_n))^{-1}\,f(\operator{A}_n)\,
  (h(\operator{A}_n))^{-1}$
  are bounded,
  operator norm $\|\operator{B}_n\|\leq 1$ and
  \begin{equation*}
    (h(\operator{A}_n))^{-1}\,f(\operator{A}_n)\,
    (h(\operator{A}_n))^{-1}
    \xrightarrow{s}
    (h(\operator{A}))^{-1}\,
    f(\operator{A})\,(h(\operator{A}))^{-1}.
  \end{equation*}
  Because $\phi\in \vectorspace{D}(h(\operator{A}))$, the result follows.
\end{proof}
\begin{remark}
  This theorem suggests us a new form of convergent quadrature rule.
  If we can compute
  the Fourier
  coefficients $\tilde{v}_i=\langle \phi_i, h\rangle$, using the above theorem, we can
  construct a sequence of convergent quadrature rules
  \begin{equation*}
    \int_\Omega f(g(\mat{x}))\,w(\mat{x})\,d\mat{x}\approx
    \tilde{\mat{v}}_n^*\,f(\mat{M}_n[g])\,\tilde{\mat{v}}_n,
  \end{equation*}
  where vector $\mat{\tilde{v}}_n=(h(\mat{M}_n[g]))^{-1}\,\mat{\tilde{v}}$,
  and components
  of vector $\mat{\tilde{v}}$ are $\tilde{v}_i$ respectively.
  The nodes are then the eigenvalues
  $\lambda_j$ of each finite matrix approximation $\mat{M}_n[g]$ and the weights are
  \begin{equation}
    \label{eq:general-weights}
    w_j=\frac{\abs{\mat{\tilde{v}}^*\,\mat{\tilde{u}}_j}^2}{h^2(\lambda_j)},
  \end{equation}
  where
  $\mat{\tilde{u}}_j$ are the unit length eigenvectors of $\mat{M}_n[g]$
  corresponding to the eigenvalues
  $\lambda_j$ respectively. The numerical integral of
  \eqref{eq:numint-def} is a special case
  of this approach with $h=1$. We give a numerical example of the effects
  of this theorem
  in practice in Section~\ref{sec:numerical-example}.
\end{remark}

\subsection{Convergence for quadratically bounded functions}
\label{sec:convergence-for-quadratically-bounded-functions}
For convergence without an approximating sequence for
vectors $\phi\in\vectorspace{D}(f(\operator{A}))$, we can prove 
a limited form of convergence for a limited class of unbounded
functions. 

The following theorem shows that if we can show convergence for one function,
it will follow for all smaller functions.
\begin{theorem}
  \label{thm:scalar-dominated-convergence}
  Let $\operator{A}_n$ and $\operator{A}$ be self-adjoint operators
  such that $\operator{A}_n\xrightarrow{srs}\operator{A}$. Let
  $f$ and $h$ be such unbounded functions that they are
  Riemann--Stieltjes integrable with respect
  to $\operator{A}$,
  $\abs{f(x)}\leq \abs{h(x)}$ for all $x\in\mathbb{R}$ and
  $f$ and $h$ are continuous at eigenvalues of $\operator{A}$.
  Then
  for any $\phi\in\vectorspace{D}(h(\operator{A}))$
  \begin{equation*}
    \langle \phi, h(\operator{A}_n)\,\phi\rangle\rightarrow 
    \langle \phi, h(\operator{A})\,\phi\rangle
    \Rightarrow
    \langle \phi, f(\operator{A}_n)\,\phi\rangle\rightarrow 
    \langle \phi, f(\operator{A})\,\phi\rangle
      ~as~ n\rightarrow\infty.
  \end{equation*}
\end{theorem}
\begin{proof}
  This follows from \cite[Th\'eor\`eme~3]{Jouravsky:1928}.
\end{proof}

The following lemma makes it possible to show not only the convergence of the 
norms of the vectors but the vectors themselves as well.
\begin{lemma}
  \label{lemma:convergence-of-norm}
  Let self-adjoint operators $\operator{A}_n$ converge in the strong resolvent 
  sense to self-adjoint operator $\operator{A}$. Let 
  function 
  $h:~\mathbb{R}\mapsto\mathbb{C}$ and
  vector
  $\phi\in\vectorspace{D}(h(\operator{A}))
  \cap_{n=0}^\infty\vectorspace{D}(h(\operator{A}_n))$.
  If 
  $\|h(\operator{A}_n)\,\phi\|\rightarrow\|h(\operator{A})\,\phi\|$, then
  also $h(\operator{A}_n)\,\phi\rightarrow h(\operator{A})\,\phi$
  as $n\rightarrow\infty$.
\end{lemma}
\begin{proof}
  Let $\operator{E}(t)$ and $\operator{E}_n(t)$ be the spectral families
  of $\operator{A}$ and $\operator{A}_n$, respectively.
  We can select an 
  interval $I$ with endpoints that are not eigenvalues of $\operator{A}$
  such that
  \begin{equation}
    \label{eq:norm-in-tail}
    \int_{\mathbb{R}\setminus I} \abs{h(t)}^2\,d\|\operator{E}(t)\,\phi\|^2<
    \frac{\epsilon^2}{16}=\epsilon_1^2.
  \end{equation}
  Because $\operator{A}_n\xrightarrow{srs}\operator{A}$, 
  for all $\epsilon > 0$, there is an $N_1$ such that for all $k>N_1$,
  we have
  \begin{equation}
    \label{eq:convergence-of-norm-in-center}
    \left\lvert
    \int_I \abs{h(t)}^2\,d\|\operator{E}(t)\,\phi\|^2-
    \int_I \abs{h(t)}^2\,d\|\operator{E}_k(t)\,\phi\|^2
    \right\rvert < \frac{\epsilon^2}{8}=\epsilon_2^2
  \end{equation}
  and
  \begin{equation}
    \label{eq:convergence-in-center}
    \left\|
    \int_I h(t)\,d\operator{E}(t)\,\phi-\int_Ih(t)\,d\operator{E}_k(t)\,\phi
    \right\|<\frac{\epsilon}{4}=\epsilon_3.
  \end{equation}
  The convergence of 
  $\|h(\operator{A}_n)\,\phi\|\rightarrow\|h(\operator{A})\,\phi\|$ means that
  for all $\epsilon > 0$, there is an $N_2$ such that for all $k>N_2$, we have
  \begin{align*}
    \epsilon_4=\frac{\epsilon^2}{16}
    >&
    \abs{\|h(\operator{A})\,\phi\|^2-\|h(\operator{A}_k)\,\phi\|^2}\\
    =&
    \left\lvert
    \int_I\abs{h(t)}^2\,d\|\operator{E}(t)\,\phi\|^2+
    \int_{\mathbb{R}\setminus I}|h(t)|^2\,d\|\operator{E}(t)\,\phi\|^2
    \right.\\
    &
    \left.
    -\int_I|h(t)|^2\,d\|\operator{E}_k(t)\,\phi\|^2-
    \int_{\mathbb{R}\setminus I}\abs{h(t)}^2\,d\|\operator{E}_k(t)\,\phi\|^2
    \right\rvert,
  \end{align*}
which means by \eqref{eq:norm-in-tail} and
\eqref{eq:convergence-of-norm-in-center} that
\begin{equation}
  \label{eq:convergence-of-norm-in-tail}
  \int_{\mathbb{R}\setminus I}\abs{h(t)}^2\,d\|\operator{E}_k(t)\,\phi\|^2<
  \epsilon_1^2+\epsilon_2^2 +\epsilon_4^2 =\frac{\epsilon^2}{4}=\epsilon_5
\end{equation}
for all $k>M=\max(N_1,N_2)$.
By \eqref{eq:norm-in-tail}, \eqref{eq:convergence-in-center}, and
  \eqref{eq:convergence-of-norm-in-tail} we have
\begin{align*}
  &\|(h(\operator{A})-h(\operator{A}_k))\,\phi\|\\
  &=
  \left\|
  \int_{\mathbb{R}\setminus I} h(t)\,d\operator{E}(t)\,\phi+
  \int_I h(t)\,d\operator{E}(t)\,\phi
  -\int_{\mathbb{R}\setminus I}h(t)\,d\operator{E}_k(t)\,\phi
  -\int_I h(t)\,d\operator{E}_k(t)\,\phi
  \right\|\\
  &\leq
  \left\|
  \int_Ih(t)\,d\operator{E}(t)\,\phi-\int_Ih(t)\,d\operator{E}_k(t)\,\phi
  \right\|
  +\sqrt{\int_{\mathbb{R}\setminus I}\abs{h(t)}^2\,d\|\operator{E}(t)\,\phi\|^2}\\
  &~~~+\sqrt{\int_{\mathbb{R}\setminus I}\abs{h(t)}^2\,d\|\operator{E}_k(t)\,\phi\|^2}\\
  &<
  \epsilon_3+\epsilon_1+\epsilon_5=\epsilon
\end{align*}
for all $k>M$.
\end{proof}

\begin{theorem}
  \label{thm:dominated-convergence}
  Let $\operator{A}_n$ and $\operator{A}$ be self-adjoint operators
  such that $\operator{A}_n\xrightarrow{srs}\operator{A}$. Let
  $f$ and $h$ be such unbounded functions that they are
  Riemann--Stieltjes integrable with respect
  to $\operator{A}$,
  $\abs{f(x)}\leq \abs{h(x)}$ for all $x\in\mathbb{R}$ and
  $f$ and $h$ are continuous at eigenvalues of $\operator{A}$.
  Then
  for any $\phi\in\vectorspace{D}(h(\operator{A}))$
  \begin{equation*}
    h(\operator{A}_n)\,\phi\rightarrow h(\operator{A})\,\phi
    \Rightarrow
    f(\operator{A}_n)\,\phi\rightarrow f(\operator{A})\,\phi
  ~as~ n\rightarrow\infty.
  \end{equation*}
\end{theorem}
\begin{proof}
This follows from Theorem~\ref{thm:scalar-dominated-convergence} and
  Lemma~\ref{lemma:convergence-of-norm} by using 
  identity 
  $\|f(\operator{A})\,\phi\|^2= \langle \phi, \abs{f(\operator{A})}^2\,\phi \rangle$
  for functions $f$ and $h$.
\end{proof}

For example, we can apply this theorem to linearly bounded functions.
\begin{theorem}
  \label{thm:linearly-bounded-convergence}
  Let operator $\operator{A}$ be self-adjoint with an infinite matrix
  representation $\mat{A}_{\infty}$ and vector $\mat{v}\in\ell^2_0$.
  Let finite approximations be
 $   \left[
      \mat{A}_n
      \right]_{i,j}
    =
    \left[
      \mat{A}_\infty
      \right]_{i,j}$ and
    $
    [\mat{v}_n]_i
    =
    [\mat{v}]_i$
  for $i,j=0,1,\ldots,n$. 
  Let function $f$ be Riemann--Stieltjes integrable with respect to
  $\operator{A}$ and
  linearly bounded, that is, for all $x\in\mathbb{R}$,
  $\abs{f(x)}\leq a + b\,\abs{x}$
  for some positive $a$ and $b$.  
  Then,   as $n\rightarrow\infty$,
  \begin{equation*}
    \left[
      \begin{array}{c}
        f(\mat{A}_n)\,\mat{v}_n\\
        \mat{0}_{\infty\times 1}
      \end{array}
      \right]
    \rightarrow
    f
    \left(
    \mat{A}_\infty
    \right)
    \,\mat{v}.
  \end{equation*}
\end{theorem}
\begin{proof}
  By Theorem~\ref{thm:srs-convergence}, 
  convergence holds for function $h(x) = x$, and
  by Theorem~\ref{thm:dominated-convergence}, 
  for function $h(x)=\abs{x}$. Convergence then holds
  also for function $h(x) = b\,\abs{x}$, for $h(x) = a+b\,\abs{x}$,
  and again by Theorem~\ref{thm:dominated-convergence}, for $f(x)$.
\end{proof}
We can apply this theorem to a product of two linearly bounded
functions.
\begin{theorem}
  \label{thm:product-convergence}
  Let self-adjoint operators $\operator{A}$ and $\operator{B}$ have
  infinite matrix representations $\mat{A}_\infty$ and $\mat{B}_\infty$.
  Furthermore, let vectors $\mat{u},\mat{v}\in\ell^2_0$ and 
  finite approximations have the elements
  \begin{equation*}
    \begin{array}{cccc}
      [\mat{A}_n]_{i,j}=[\mat{A}_\infty]_{i,j},&
      [\mat{B}_n]_{i,j}=[\mat{B}_\infty]_{i,j},&
      [\mat{u}_n]_i  = [\mat{u}]_i,&
      [\mat{v}_n]_i = [\mat{v}]_i
    \end{array}
  \end{equation*}
  for $i,j=0,1,\ldots,n$.
  Then for linearly bounded $f_1,f_2$ that are Riemann--Stieltjes
  integrable with respect to $\operator{A}$ and $\operator{B}$ respectively,
      as $n\rightarrow\infty$, we have
  \begin{equation*}
      \mat{u}_n^*\,
      f_1\left(
      \mat{A}_n
      \right)\,
      f_2\left(
      \mat{B}_n
      \right)\,
      \mat{v}_n
      \rightarrow
      \mat{u}^*\,f_1(\mat{A}_\infty)\,f_2(\mat{B}_\infty)\,\mat{v}.
  \end{equation*}
\end{theorem}
\begin{proof}
  This follows from 
  Theorem~\ref{thm:linearly-bounded-convergence} and
  the property that in a Hilbert space, if
  $\mat{x}_n\rightarrow \mat{x}$ and $\mat{y}_n\rightarrow \mat{y}$ then
  $\langle \mat{x}_n,\mat{y}_n\rangle \rightarrow \langle \mat{x},\mat{y}\rangle$
  \cite[p.~199]{Riesz+Nagy:1955} or \cite[Lemma~3.2-2]{Kreyszig:1989}.
\end{proof}

We can now apply this theorem to quadratically bounded functions.
\begin{theorem}
  \label{thm:quadratic-convergence}
  Let self-adjoint operator $\operator{A}$ have
  infinite matrix representation $\mat{A}_\infty$.
  Let vectors $\mat{u},\mat{v}\in\ell^2_0$, and 
  let finite approximations have the
  elements
  \begin{equation*}
    \begin{array}{ccc}
    [\mat{A}_n]_{i,j}=[\mat{A}_\infty]_{i,j},&
    [\mat{u}_n]_i  = [\mat{u}]_i,&
    [\mat{v}_n]_i = [\mat{v}]_i
    \end{array}
  \end{equation*}
  for $i,j=0,1,\ldots,n$.
  Let function $f$ be Riemann--Stieltjes integrable
  with respect to $\operator{A}$
  and quadratically bounded, that is,
  for all $x\in\mathbb{R}$,
    $\abs{f(x)}\leq a + b\,x^2$, for some positive $a$ and $b$. Then
  $
      \mat{u}_n^*\,
      f\left(
      \mat{A}_n
      \right)
      \,
      \mat{v}_n
      \rightarrow
      \mat{u}^*\,f(\mat{A}_\infty)\,\mat{v}
      $
      as $n\rightarrow\infty$.
\end{theorem}
\begin{proof}
  This follows from Theorem~\ref{thm:product-convergence} by selecting
  $\operator{A}=\operator{B}$,
  $f_1=\operatorname{sgn}(f)\,\sqrt{\abs{f}}$ and $f_2=\sqrt{\abs{f}}$.
  The product function $f_1\,f_2$ is Riemann--Stieltjes
  integrable by \cite[Theorem~6.11 and 6.13]{Rudin:1976}.
\end{proof}

For the convergence of the numerical approximation of integrals
as in \eqref{eq:numint-def} and \eqref{eq:product-convergence}, we
then have the following theorem.
\begin{theorem}
  \label{thm:quadratic-integral-convergence}
  Let orthonormal functions $\phi_0=1,\phi_1,\phi_2,\ldots$ be dense in
  $\vectorspace{L}^2_w(\Omega)$.
  Let function
  $g$ satisfy \eqref{eq:mult-mat} and 
  let matrices $\mat{M}_n[g]$ have
  elements as in \eqref{eq:matrix-elements}.
  Further, let quadratically bounded function $f$ be
  Riemann--Stieltjes integrable with respect to $\operator{M}[g]$.
  Then
  \begin{equation*}
    \lim_{n\rightarrow\infty} [f(\mat{M}_n[g])]_{0,0}=
    \int_{\Omega} f(g(\mat{x}))\,w(\mat{x})\,d\mat{x}.
  \end{equation*}
  Similarly, for functions $g_1$ and $g_2$ satisfying \eqref{eq:mult-mat}
  and linearly bounded $f_1$ and $f_2$, we have
  \begin{equation*}
    \lim_{n\rightarrow\infty} [f_1(\mat{M}_n[g_1])\,f_2(\mat{M}_n[g_2])]_{0,0}=
    \int_{\Omega} f_1(g_1(\mat{x}))\,f_2(g_2(\mat{x}))\,w(\mat{x})\,d\mat{x}.
  \end{equation*}
\end{theorem}
\begin{proof}
The result follows from Theorems 
\ref{thm:product-convergence} and \ref{thm:quadratic-convergence}.
\end{proof}
We can extend this result to polynomially bounded functions in the
special case that the inside function $g$ and the basis
functions are polynomials.
We define the total degree of a multivariate monomial
$\prod_{i=0}^d x_i^{n_i}$ as $\sum_{i=0}^d n_i$. Then for a polynomial,
that is, a finite linear combination of monomials, its total degree is the
total degree of the highest monomial.
If polynomials are dense in a Hilbert space, they can
be partially ordered according to the total degree of the polynomials, 
for example, in graded lexicographic order 
\cite[Chapter~3.1]{Dunkl+Xu:2014}.
\begin{theorem}
  Let the orthonormal polynomials $\phi_0,\phi_1,\phi_2,\ldots$ be partially
  ordered by their total degree and dense in
  $\vectorspace{L}^2_w(\Omega)$.
  Let the function $g$ be polynomial and 
  let the matrices $\mat{M}_n[g]$ have
  elements as in \eqref{eq:matrix-elements}. Let a
  function $f$ be polynomially bounded
  that is,
  $f(x)\leq a + b\,\abs{x}^m$ for some positive $a,b$ and $m\in\mathbb{N}$, then
  \begin{equation*}
    \lim_{n\rightarrow\infty} [f(\mat{M}_n[g])]_{0,0}=
    \int_{\Omega} f(g(\mat{x}))\,w(\mat{x})\,d\mat{x},
  \end{equation*}
  if $f$ is  Riemann--Stieltjes integrable with respect to $\operator{M}[g]$.
\end{theorem}
\begin{proof}
  Because polynomials are dense in   $\vectorspace{L}^2_w(\Omega)$, the function
  $a+b\,\abs{g(\mat{x})}^m$ is integrable,
  and $g$ satisfies \eqref{eq:mult-mat}.
  Let the total degree of polynomial $g$ be $k$.
  Let the projection operator $\operator{P}_n$ be defined by
   $ \operator{P}_n\,\psi=\sum_{i=0}^n\langle \phi_i,\psi\rangle\,\phi_i$.
  For a polynomial $\phi_j$ of total degree $p$, we can
  select $n$ so that the linear combination of
  $\{\phi_i\}_{i=0}^n$ covers all polynomials of
  total degree $k\,m + p$ and then we have the following equalities
  \begin{align*}
    g^m\,\phi_j = \operator{M}[g]^m\,\phi_j
    &=
    \left(
    \operator{P}_n\,\operator{M}[g]\,\operator{P}_n
    \right)^m\,\phi_j\\
    \mat{M}_\infty[g]^m\,\mat{e}_j
    &=
    \left[
      \begin{array}{cc}
        \mat{M}_n[g]             & \mat{0}_{(n+1)\times\infty} \\
        \mat{0}_{\infty\times (n+1)} & \mat{0}_\infty      
      \end{array}
      \right]^m\,\mat{e}_j.
  \end{align*}
  By Theorem~\ref{thm:dominated-convergence}, we have
  \begin{equation*}
    \left[
      \begin{array}{cc}
        f(\mat{M}_n[g])             & \mat{0}_{(n+1)\times\infty} \\
        \mat{0}_{\infty\times (n+1)} & \mat{0}_\infty      
      \end{array}
      \right]\,\mat{e}_j
    \rightarrow
      f(\mat{M}_\infty[g])\,\mat{e}_j,
  \end{equation*}
  from where it follows that
  \begin{equation*}
  [f(\mat{M}_n[g])]_{0,0}=\mat{e}_0^\top\,f(\mat{M}_n[g])\,\mat{e}_0\rightarrow
  \mat{e}_0^\top\,f(\mat{M}_\infty[g])\,\mat{e}_o=
  \int_\Omega f(g(\mat{x}))\,w(\mat{x})\,d\mat{x}.
  \end{equation*}
\end{proof}
\begin{remark}
  The same arguments hold also more generally for the convergence of
  $\mat{e}_i^\top\,f(\mat{A}_n)\,\mat{e}_j\rightarrow \mat{e}_i^\top\,f(\mat{A}_\infty)\,\mat{e}_j$
  when $\mat{A}_\infty$ is a self-adjoint infinite sparse matrix, that is,
  an infinite matrix with only a finite number of non-zero components in each
  column and row.
\end{remark}
\begin{remark}
  In the special case when the inside function is
  $g(x)=\operatorname{id}(x)=x$, we have essentially the same proof as 
  in \cite[Sections~1.8~and~1.13]{Simon:2009} for the convergence of
  Gaussian quadrature on 
  polynomially bounded functions.
\end{remark}

We can use the following lemma to prove convergence beyond polynomially bounded
functions.
\begin{lemma}
  \label{lemma:bounded-convergence}
  Let self-adjoint operators $\operator{A}_n$ converge to a self-adjoint
  operator $\operator{A}$ in the strong resolvent sense. Let 
  a function 
  $f:~\mathbb{R}\mapsto [0,\infty)$ and a vector $\phi$ be such that
  $\langle \phi,f(\operator{A})\,\phi\rangle<\infty$.
  If 
  $\langle \phi, f(\operator{A}_n)\,\phi\rangle$ is bounded by   
  $\langle \phi,f(\operator{A})\,\phi\rangle$, then  
  $\langle \phi, f(\operator{A}_n)\,\phi\rangle \rightarrow 
  \langle \phi, f(\operator{A})\,\phi\rangle$ as $n\rightarrow\infty$.      
\end{lemma}
\begin{proof}
Let the spectral families of $\operator{A}_n$ and $\operator{A}$ be 
$\operator{E}_n(t)$ and $\operator{E}(t)$, respectively.
Because $f$ is nonnegative, for an $\epsilon > 0$, we can select a finite 
interval $I$ with endpoints that are not eigenvalues of $\operator{A}$, so that
\begin{equation*}
  \langle \phi, f(\operator{A})\,\phi\rangle - \frac{\epsilon}{2} <
  \int_I f(t)\,d\langle \phi,\operator{E}(t)\,\phi\rangle \leq
  \langle \phi, f(\operator{A})\,\phi\rangle.
\end{equation*}
Due to the strong resolvent convergence, there is 
$N\in\mathbb{N}$, so that for all $k>N$, we have
\begin{equation*}
  \left\lvert
  \int_I f(t)\,d\langle \phi, \operator{E}(t)\,\phi\rangle -
  \int_I f(t)\,d\langle \phi, \operator{E}_k(t)\,\phi \rangle
  \right\rvert < \frac{\epsilon}{2}.
\end{equation*}
Because $f$ is nonnegative, these inequalities mean that for all $k>N$, 
we have
\begin{equation*}
  \langle \phi, f(\operator{A})\,\phi\rangle - \epsilon
  <
  \int_I f(t)\,d\langle \phi, \operator{E}_k(t)\,\phi\rangle
  \leq
  \langle \phi, f(\operator{A}_k)\,\phi \rangle
  \leq
  \langle \phi, f(\operator{A})\,\phi \rangle.
\end{equation*}
\end{proof}

In \cite[Chapter~IV, Section~8--10]{Shohat+Tamarkin:1963} and
\cite[Chapter~3, Theorem~1.6]{Freud:1971},
the proof of convergence for
  Gaussian quadratures is based on the inequality
  \begin{equation}
    \label{eq:gauss-inequality}
    \sum_{j=0}^n w_j\,h(x_j) \leq \int_\Omega h(x)\,w(x)\,dx,
  \end{equation}
  when all even derivatives of $h$ are nonnegative for all $x\in\Omega$,
  which implies
  \eqref{eq:wx_ineq} that was used in the proof of convergence in
  \cite{Uspensky:1928,Jouravsky:1928,Bultheel+Diaz-Mendoza+Gonzalez-Vera+Orive:2000}. This kind of
  inequality makes it possible to prove convergence for functions that are
  bounded by functions that can have negative odd derivatives.
  When the basis functions are polynomials,
  in terms of the infinite matrices and the leading principal submatrices,
  \eqref{eq:gauss-inequality} can be expressed as
  \begin{equation}
    \label{eq:matrix-gauss-inequality}
    [h(\mat{M}_n[\operatorname{id}])]_{0,0} \leq 
    [h(\mat{M}_\infty[\operatorname{id}])]_{0,0},
  \end{equation}
  where function $\operatorname{id}(x)=x$ but it does not hold for
  arbitrary inside functions $g\neq \operatorname{id}$ when the basis 
  functions are
  not polynomials or for
 other matrix elements except
  the $0,0$ element. Therefore, we can consider 
  functions with nonnegative even derivatives at zero as the highest possible 
  bound for which we have convergence
  if we can generalise \eqref{eq:matrix-gauss-inequality} for other
  elements than just $0,0$ element and arbitrary
  basis functions and inside function $g$. 
\subsection{Nonnegative matrix coefficients}
\label{sec:nonnegative-matrix-coefficients}
We can further expand the space of numerically integrable functions
when the matrix coefficients are all nonnegative. This can seem like a hard
restriction but, for example, the Jacobi matrix side diagonal elements
are always positive, only the diagonal elements can be negative
\cite[Theorem~1.27 and Definition~1.30]{Gautschi:2004}.
Examples of such cases are Jacobi matrices
for Chebyshev polynomials of the fourth kind and the Jacobi and the
Meixner--Pollaczek  polynomials
with parameter values on certain ranges \cite[Table~1.1]{Gautschi:2004}.
In some cases, it is also possible to extend the results for matrices with nonnegative coefficients
to matrices with positive and negative real coefficients by considering the positive
and negative parts separately.

We extend the concept of $m$-positive and $m$-monotone functions so that
matrix entries are allowed to be $\infty$, and the inequality for
a matrix element $a$ is defined so that $a\leq\infty$ for $a<\infty$,
but $\infty \leq \infty$ is not true.
\begin{definition}
  \label{def:unbounded-m-positive}
  Let $f:I\mapsto \mathbb{R}$ be a real function defined on an interval $I\subset \mathbb{R}$.
  \begin{enumerate}
  \item
    $f$ is unbounded $m$-positive if $0\in I$ and
    $f(\mat{A})\ewgeq \mat{0}_\infty$ for every
    symmetric infinite matrix $\mat{A}\ewgeq \mat{0}_\infty$
    with spectrum in $I$. 
    $\mat{A}$ and $f(\mat{A})$ are not required to satisfy 
    \eqref{eq:infinite-matrix-condition}.
    Elements of 
    $f(\mat{A})$ are allowed to be $\infty$.
  \item
    $f$ is unbounded $m$-monotone if
    $\mat{A}\ewleq \mat{B} \Rightarrow f(\mat{A})\ewleq f(\mat{B})$, for all
    symmetric infinite matrices
    $\mat{A},\mat{B}\ewgeq \mat{0}_\infty$ with spectra in $I$.
    $\mat{A}$, $\mat{B}$, $f(\mat{A})$, and $f(\mat{B})$ are not required
    to satisfy \eqref{eq:infinite-matrix-condition}.
    Elements of 
     $f(\mat{B})$ are allowed to be $\infty$.
  \end{enumerate}
\end{definition}
For example, monomials and nonnegative linear combinations of monomials
are unbounded $m$-positive functions.
We can extend Theorem~\ref{thm:m-positive-equal-m-monotone} to unbounded
$m$-positive and unbounded $m$-monotone functions.
\begin{theorem}
  \label{thm:operator-m}
  Let $f:I\mapsto \mathbb{R}$ be a real function defined on an interval
  $I\subset \mathbb{R}$ such that $0\in I$ and
  $f(\mat{0}_\infty)\ewgeq \mat{0}_\infty$. 
  Then $f$ is
  unbounded $m$-positive if and only if it is unbounded $m$-monotone.
\end{theorem}
\begin{proof}
  We adapt the proof of \cite[Theorem~2.1]{Hansen:1992} for symmetric arbitrary
  sized $n\times n$ matrices to unbounded symmetric infinite matrices
  and check that the same steps are valid.
  First, we prove that if $f$ is unbounded $m$-monotone, it is also
  unbounded $m$-positive. We set $\mat{A}=\mat{0}_\infty$ to obtain
  $\mat{0}_\infty\ewleq f(\mat{0}_\infty)\ewleq f(\mat{B})$ for all symmetric
  infinite matrices.

  To prove that unbounded $m$-positive $f$ is also unbounded $m$-monotone,
  we consider symmetric infinite matrices $\mat{A},\mat{B}\ewgeq \mat{0}$
  with the spectra in $I$. 
  We construct a $2\times 2$ block infinite matrix
  \begin{equation*}
    \mat{X}=
    \left[
      \begin{array}{cc}
        \mat{A} & \mat{0}_\infty \\
        \mat{0}_\infty & \mat{B}
      \end{array}
      \right].
  \end{equation*}
  We can extend the $\ewleq$ relation for the block matrix simply as blockwise.
  The spectrum of $\mat{X}$ is also contained in $I$.
  We define a $2\times 2$ unitary infinite block matrix
  \begin{equation*}
    \mat{U}=
    \frac{1}{\sqrt{2}}
    \left[
      \begin{array}{cc}
        \mat{I}_\infty & -\mat{I}_\infty\\
        \mat{I}_\infty & \mat{I}_\infty
      \end{array}
      \right].
  \end{equation*}
  We calculate
  \begin{equation*}
    \mat{U}^*\,\mat{X}\,\mat{U}=
    \frac{1}{2}
    \left[
      \begin{array}{cc}
        \mat{A}+\mat{B} & \mat{B} - \mat{A}\\
        \mat{B} - \mat{A} & \mat{A} + \mat{B}
      \end{array}
      \right],
  \end{equation*}
  and notice that the spectrum of
  $\mat{U}^*\,\mat{X}\,\mat{U}$ is
  contained in $I$ and that
  $\mat{U}^*\,\mat{X}\,\mat{U}\ewgeq \mat{0}_\infty$
  if $\mat{A}\ewleq\mat{B}$. If $f$ is unbounded $m$-positive, then
  \begin{align*}
    \mat{0} &\ewleq f(\mat{U}^*\,\mat{X}\,\mat{U})
    =
    \mat{U}^*\,f(\mat{X})\,\mat{U}
    =
    \mat{U}^*\,
    \left[
      \begin{array}{cc}
        f(\mat{A})    & \mat{0}_\infty\\
        \mat{0}_\infty & f(\mat{B})
      \end{array}
      \right]
    \mat{U}\\
    &=
    \frac{1}{2}
    \left[
      \begin{array}{cc}
        f(\mat{A})+f(\mat{B}) & f(\mat{B})-f(\mat{A})\\
        f(\mat{B})-f(\mat{A}) & f(\mat{A})+f(\mat{B})
      \end{array}
      \right],
  \end{align*}
  which implies that $f(\mat{A}) \ewleq f(\mat{B})$. Here we multiply only
  $f(\mat{A})$, which has finite elements, with $-\mat{I}_\infty$ and define
  $\infty-a=\infty$ for all $a\in[0,\infty)$.
\end{proof}

Absolutely monotone functions, that is, functions that have power series 
expansion $h(x)=\sum_{n=0}^\infty c_n\,x^n$ where $c_n\geq 0$, 
are unbounded $m$-positive and
therefore also unbounded $m$-monotone. These functions are also very
fast-growing, so if we have convergence for such functions, we also have
convergence for a large family of functions that grow slower.
  Examples of functions that have power series representation with
  nonnegative coefficients
  are $\exp(x)$, $\cosh(\sqrt{x})$, which were used in \cite{Uspensky:1928},
  and more generally, the Mittag--Leffler functions
  $  E_\gamma(x)=\sum_{k=0}^\infty \frac{x^k}{\Gamma(\gamma\,k+1)}$
  where $0<\gamma\leq 2$ and $x\in[0,\infty)$
  \cite[eq. (3.3)]{Bultheel+Diaz-Mendoza+Gonzalez-Vera+Orive:2000}.

For functions bounded by unbounded $m$-positive functions, we have the following 
theorem. 
\begin{theorem}
  \label{thm:m-positive-convergence}
  Let a self-adjoint infinite matrix $\mat{A}_\infty\ewgeq \mat{0}_\infty$ satisfy
  \eqref{eq:infinite-matrix-condition}. 
  Let function $f$ be bounded by a positive and unbounded $m$-positive 
  function $h$, that is, $\abs{f(x)}\leq h(x)$ for $x\in\sigma(\mat{A}_\infty)$.
  Let finite matrix approximations $\mat{A}_n$ be the leading principal
  submatrices of $\mat{A}_\infty$.
  Then, as $n\rightarrow\infty$,
  $[f(\mat{A}_n)]_{k,k}\rightarrow [f(\mat{A}_\infty)]_{k,k}$
  for all $k=0,1,2,\ldots$ for which $[h(\mat{A})]_{k,k}<\infty$.
\end{theorem}
\begin{proof}
  Let self-adjoint infinite matrices
  \begin{equation*}
    \mat{B}_n=
    \left[
      \begin{array}{cc}
        \mat{A}_n                & \mat{0}_{(n+1)\times\infty}\\
        \mat{0}_{\infty\times (n+1)} & \mat{0}_\infty 
      \end{array}
      \right].
  \end{equation*}
  Because $\mat{B}_n\ewleq \mat{A}$,
  by Theorem~\ref{thm:operator-m}, the sequence
  $[h(\mat{B}_n)]_{k,k}$ is bounded by $[h(\mat{A})]_{k,k}$, which means by
  Lemma~\ref{lemma:bounded-convergence} that it must converge to
  $[h(\mat{A})]_{k,k}$. 
  Because for $n\geq k$, we have 
  $[h(\mat{A}_n)]_{k,k}=[h(\mat{B}_n)]_{k,k}$, we also have
  $[h(\mat{A}_n)]_{k,k}\rightarrow [h(\mat{A})]_{k,k}$.
  Due to Theorem~\ref{thm:scalar-dominated-convergence}, we have
  convergence for $f$ bounded by $h$. 
\end{proof}

In terms of restrictions for the bounds of the outside function $h$,
we can reach the
most generic level of convergence, that is, convergence for functions
that have nonnegative even derivatives, like in
\cite[Chapter~IV, Section~8--10]{Shohat+Tamarkin:1963},
on the interval $(-\infty,\infty)$
when $w$ is an even function, basis functions are even or odd,
and $g$ is odd. In more general cases,
we have to use a stricter condition that the bounding function $h$
is absolutely monotone, that is, has all derivatives positive.
\begin{theorem}
  \label{thm:power-series-convergence}
  Let the orthonormal functions $\phi_0,\phi_1,\phi_2,\ldots$ be dense in
  $\vectorspace{L}^2_w(\Omega)$.
  Let the function
  $g$ satisfy \eqref{eq:mult-mat} and let the matrices $\mat{M}_n[g]$ have
  elements as in \eqref{eq:matrix-elements} that are nonnegative.
  Furthermore, let a function $f$ be Riemann--Stieltjes integrable with
  respect to multiplication operator $\operator{M}[g]$ and bounded by
  a function $h$ so that $\abs{f(x)}\leq h(x)$ and
    \begin{equation*}
      \int_{\Omega} h(g(\mat{x}))\,\abs{\phi_i(\mat{x})}^2\,w(\mat{x})\,d\mat{x} < \infty.
    \end{equation*}  
  Then
  \begin{equation*}
    \lim_{n\rightarrow\infty} [f(\mat{M}_n[g])]_{i,i}=
    \int_{\Omega} f(g(\mat{x}))\,\abs{\phi_i(\mat{x})}^2\,w(\mat{x})\,d\mat{x},
  \end{equation*}
  if one of the following additional conditions holds:
  \begin{enumerate}
  \item Function $h$ has the form
     $ h(x) = \sum_{k=0}^\infty a_k\,x^k$
    where $a_k\geq 0$ for all $k=0,1,\ldots$.
  \item Function $g(\mat{x})$ satisfies
    \begin{equation*}
      \int_\Omega g(\mat{x})^{2\,k+1}\,\abs{\phi_i(\mat{x})}^2\,w(\mat{x})\,d\mat{x}
      =[\mat{M}_n[g]^{2\,k+1}]_{i,i}=0
    \end{equation*}
    for all $k=0,1,2,\ldots$ and all even coefficients of 
 $     h(x)=\sum_{k=0}^\infty a_k\,x^k$
    are nonnegative: $a_{2k}\geq 0$.
  \end{enumerate}
\end{theorem}
\begin{proof}
  The proof follows from
  Theorem~\ref{thm:m-positive-convergence}. Under the second condition,
  the odd $a_{2k+1}$ coefficients can also be negative.
\end{proof}

We give a numerical example demonstrating the difference in convergence
for matrices with purely nonnegative elements and matrices with
mixed sign elements in Section~\ref{sec:numerical-example}.

In some cases, it is also possible to prove convergence when
matrix coefficients have mixed signs by considering the negative and positive
parts separately. We decompose a self-adjoint real infinite matrix $\mat{A}$
into positive $\mat{P}$
and negative $\mat{N}$ parts with
\begin{align*}
  [\mat{A}^+]_{i,j}&=
  \left\{
    \begin{array}{llll}
      [\mat{A}]_{i,j} &\mathrm{for} & [\mat{A}]_{i,j}\geq 0,\\
      0              &\mathrm{for} & [\mat{A}]_{i,j}<0,
    \end{array}
    \right.
\end{align*}
so that $\mat{P}=\mat{A}^+$ and $\mat{N}=(-\mat{A})^+$.
Then $\mat{A}=\mat{P} - \mat{N}$. We construct
a $2\times 2$ block infinite matrix with nonnegative elements
\begin{equation}
  \label{eq:positive-negative-block}
  \mat{T}=
  \left[
    \begin{array}{cc}
      \mat{P} & \mat{N}\\
      \mat{N} & \mat{P}
    \end{array}
    \right].
\end{equation}
We can express any
matrix element of a function of a self-adjoint real infinite matrix as a
function of a $2\times 2$ block infinite matrix with nonnegative coefficients
as
\begin{equation*}
  [f(\mat{A})]_{i,j}=
  \mat{e}_i^\top\,f(\mat{A})\,\mat{e}_j=
  \left[
    \begin{array}{cc}
      \mat{e}_i^\top & \mat{0}_{1\times\infty}
      \end{array}
    \right]\,
  f\left(
  \left[
    \begin{array}{cc}
      \mat{P} & \mat{N}\\
      \mat{N} & \mat{P}
    \end{array}
    \right]
  \right)
  \,\left[
    \begin{array}{c}
      \mat{e}_j\\
      -\mat{e}_j
    \end{array}
    \right].
\end{equation*}

For example,
we consider an unbounded function $g(x)=x\,\cos(\sqrt{x})$, a weight
function $w(x)=e^{-x}$ on interval $[0,\infty)$ and polynomials as orthogonal
functions.
  The infinite matrix $\mat{M}_\infty[g]$ is unbounded because function
  $g$ is an unbounded function on the interval $[0,\infty)$. The elements of
    $\mat{M}_\infty[g]$ are given by
\begin{equation*}
  \mat{M}_\infty[g]
  =
  \frac{\mathrm{i}\sqrt{\pi}\,\operatorname{erf}
    \left(
    \frac{\mathrm{i}}
         {2}
         \right)
         \,e^{-\frac{1}{4}}}
       {8192}\,
  \left[
    \begin{smallmatrix}
      5120   & 5376   & -1120  & \ldots \\
      5376   & 13632  & 4456   & \ldots \\
      -1120  & 4456   & 10293  & \ldots \\
      \svdots & \svdots & \svdots & \sddots
    \end{smallmatrix}
    \right]%
     +
  \frac{1}{4096}\,
  \left[
    \begin{smallmatrix}
      3072   & 768    & {-2208}  & \ldots\\
      768    & 192    & -6312  & \ldots\\
      -2208  & -6312  & -11613 & \ldots \\
      \svdots & \svdots & \svdots & \sddots
    \end{smallmatrix}
    \right],
\end{equation*}
where $\operatorname{erf}(x)=\frac{2}{\sqrt{\pi}}\int_0^x e^{-t^2}dt$.
Some of the elements are negative, and some are positive.
The pattern for the signs of the matrix elements is
\begin{equation*}
  \left[
    \begin{smallmatrix}
      + & - & - & + & - & + & \ldots\\
      - & - & - & - & + & - & \ldots\\
      - & - & - & - & + & + & \ldots\\
      + & - & - & - & - & + & \ldots\\
      - & + & + & - & - & - & \ldots\\
      + & - & + & + & - & - & \ldots\\
      \svdots & \svdots & \svdots & \svdots & \svdots & \svdots & \sddots
    \end{smallmatrix}
    \right].
\end{equation*}
However, it is possible to prove convergence for polynomially bounded 
functions, as is shown
in the following theorem.
\begin{theorem}
  Let $w(x)=e^{-x}$ on interval $[0,\infty)$,
    the orthogonal basis functions $\phi_i$ polynomials
and $g(x)=x\,\cos(\sqrt{x})$. Let $\mat{M}_n[g]$ have
elements as in  \eqref{eq:matrix-elements}. Then
\begin{equation*}
  [f(\mat{M}_n[g])]_{i,j}\rightarrow [\mat{M}_\infty[f(g)]]_{i,j}=
  \int_0^\infty \phi_i(x)\,\phi_j(x)\,f(g(x))\,w(x)\,dx,
\end{equation*}
when $f$ is bounded by some polynomial.
\end{theorem}
\begin{proof}
The orthogonal polynomials $\phi_i$ are known as the Laguerre polynomials.
Elements of the Jacobi matrix, the infinite matrix representation for
the function $\operatorname{id}(x)=x$, are
nonnegative. We can see this by looking at the
exact formula for the recurrence coefficients of the Jacobi matrix
for the Laguerre polynomials
\cite[Table~1.1 and Definition~1.30]{Gautschi:2004}.
The main diagonal elements are the odd positive numbers, and
  the off-diagonal
  elements are the natural numbers all in a growing order and thus nonnegative
  as is also visible in the few first matrix
  coefficients
  \begin{equation}
    \label{eq:Laguerre-system-Jacobi-matrix}
    \mat{M}_4[\operatorname{id}]=
    \left[
      \begin{array}{ccccc}
        1  &   1  &   0  &   0  &   0\\
        1  &   3  &   2  &   0  &   0\\
        0  &   2  &   5  &   3  &   0\\
        0  &   0  &   3  &   7  &   4\\
        0  &   0  &   0  &   4  &   9
      \end{array}
      \right].
  \end{equation}
  
  For a function $h(x)=x\,\cosh(\sqrt{x})$, the elements of an infinite
  matrix $\mat{M}_\infty[h]$ are also nonnegative because $h$
  has a Taylor series expansion
  $%
    h(x)=x\,\cosh(\sqrt{x})=\sum_{k=0}^\infty \frac{1}{(2\,k)!}\,x^{k+1}
  $ %
  with only nonnegative coefficients,
  and that converges for all $x\in[0,\infty)$.
    The infinite matrix representation of the multiplication operator
    of function $h$ has a series representation for the elements as
    \begin{align*}
      [\mat{M}[h]]_{i,j}
      &=
      \sum_{k=0}^\infty \frac{1}{(2\,k)!}\,
          [\mat{M}_\infty[\operatorname{id}]^{k+1}]_{i,j} \geq 0.
    \end{align*}
    The function $g$ has almost the same Taylor series expansion as
    the function $h$
    except some coefficients
    are negative
$%
  g(x)=x\,\cos(\sqrt{x})=\sum_{k=0}^\infty \frac{(-1)^k}{(2\,k)!}\,x^{k+1}.
$ %
However, the negative coefficients are bounded by the coefficients of
the Taylor series of the function $h$.

We decompose $\mat{M}_\infty[g]$ into the positive and negative parts as
$\mat{M}_\infty[g]=\mat{P}-\mat{N}$ and form a $2\times 2$
block infinite matrix $\mat{T}$ as in \eqref{eq:positive-negative-block}.
The elements of both the positive $\mat{P}$ and
the negative $\mat{N}$ part are bounded by the elements of the infinite matrix
representation
of the multiplication operator for the function $h$ as
\begin{align*}
  [\mat{P}]_{i,j}
  &=
  \sum_{k=0}^\infty \frac{1}{(4\,k)!}\,
      [\mat{M}_\infty[\operatorname{id}]^{2\,k+1}]_{i,j} \leq
      [\mat{M}_\infty[h]]_{i,j},\\
  [\mat{N}]_{i,j}
  &=
  \sum_{k=0}^\infty \frac{1}{(4\,k + 2)!}\,
      [\mat{M}_\infty[\operatorname{id}]^{2\,k+2}]_{i,j} \leq
      [\mat{M}_\infty[h]]_{i,j},
\end{align*}
that is, $\mat{P},\mat{N}\ewleq \mat{M}_\infty[h]$.
By Theorem~\ref{thm:operator-m},
\begin{equation*}
  \mat{T}^m=
  \left[
    \begin{array}{cc}
      \mat{P} & \mat{N}\\
      \mat{N} & \mat{P}
    \end{array}
    \right]^m
  \ewleq
  \left[
    \begin{array}{cc}
      \mat{M}_\infty[h] & \mat{M}_\infty[h] \\
      \mat{M}_\infty[h] & \mat{M}_\infty[h]
      \end{array}
      \right]^m.
\end{equation*}
Because $\mat{M}_\infty[h]^m$ satisfies \eqref{eq:infinite-matrix-condition}
for all $m\in\mathbb{N}$,
then also $\mat{T}^m$
satisfies \eqref{eq:infinite-matrix-condition} for all $m\in\mathbb{N}$.
By Theorem~\ref{thm:m-positive-convergence}, we have convergence for
the diagonal elements of the $2\,m$ powers. By 
Lemma~\ref{lemma:convergence-of-norm}, we have
\begin{equation*}
  \left[
    \begin{array}{cc}
      \mat{e}_i^\top & \mat{0}_{1\times\infty}
    \end{array}
    \right]
  \,
    \left[
    \begin{array}{cc}
      \mat{P}_n & \mat{N}_n\\
      \mat{N}_n & \mat{P}_n
    \end{array}
    \right]^m\rightarrow
      \left[
    \begin{array}{cc}
      \mat{e}_i^\top & \mat{0}_{1\times\infty}
    \end{array}
    \right]
    \,
    \left[
    \begin{array}{cc}
      \mat{P} & \mat{N}\\
      \mat{N} & \mat{P}
    \end{array}
    \right]^m
\end{equation*}
as $n\rightarrow\infty$.
Convergence for polynomially
bounded functions follows from
Theorem~\ref{thm:dominated-convergence}.
\end{proof}

\subsection{Improper integrals on finite endpoints of an interval}
\label{sec:improper-integrals-on-finite-endpoints-of-an-interval}
In this section, we analyse convergence for integrals that are
singular on a finite endpoint. A minimum requirement for convergence
is that a singular endpoint of the function is not a node of
the numerical integration rule. We already know that the nodes
of numerical integration, that is, the eigenvalues of
$\mat{M}_n[g]$, lie in the closed interval
$[\operatorname{ess\,inf} g,\operatorname{ess\,sup} g]$
\cite[Theorem~1]{Sarmavuori+Sarkka:2019}.
The \textit{essential infimum}
$\operatorname{ess\,inf} g$ and the
\textit{essential supremum} $\operatorname{ess\,sup} g$ are
the endpoints of the convex hull of the essential range
$\mathcal{R}(g)$ \eqref{eq:essential-range}.

In the case of an improper integral, where the integrand is singular
at an endpoint of integration, we do not want the endpoint to be
an evaluation point for the numerical integration, that is,
if the outside function $f$ is singular at $\operatorname{ess\,inf} g$
or $\operatorname{ess\,sup} g$, we do not want $\operatorname{ess\,inf} g$
or $\operatorname{ess\,sup} g$ to be an eigenvalue of $\mat{M}_n[g]$.
Fortunately, we have the following theorem.
\begin{theorem}
  \label{Thm:endpoints}
For a measurable real function $g$,
$\operatorname{ess\,inf} g$ (or $\operatorname{ess\,sup} g$) 
cannot be an eigenvalue of
the matrix $\mat{M}_n[g]$ unless
\begin{enumerate} 
\item $\operatorname{ess\,inf} g$ (or $\operatorname{ess\,sup} g$) is an eigenvalue of $\mathsf{M}[g]$ and
\item the corresponding eigenfunction is $\phi=\sum_{i=0}^n c_i\,\phi_i$. 
\end{enumerate}
\end{theorem}
\begin{proof}
We define a function $\tilde{g}=g-\operatorname{ess\,\inf} g$. 
Now $\operatorname{ess\,\inf} \tilde{g}=0$.
If 0 is an eigenvalue of $\mat{M}_n[\tilde{g}]$, then
for some non-zero $\mat{c}\in\mathbb{C}^{n+1}$, we have
$
0=\mat{c}^*\,\mat{M}_n[\tilde{g}]\,\mat{c}
$.
This is possible only if
\begin{equation*}
0=\int_{\Omega}\tilde{g}(\mat{x})\, 
\left\lvert
\sum_{i=0}^n c_i\,\phi_i(\mat{x})
\right\rvert^2\,w(\mat{x})\,d\mat{x},
\end{equation*}
which is possible only if 0 is an eigenvalue of $\mathsf{M}[\tilde{g}]$ and
$\phi=\sum_{i=0}^n c_i\,\phi_i$ 
is the corresponding eigenfunction. In that case,
$\operatorname{ess\,\inf} g$ is an eigenvalue of $\mathsf{M}[g]$ and
$\phi=\sum_{i=0}^n c_i\,\phi_i$ the corresponding eigenfunction.
The proof for $\operatorname{ess\,sup} g$ is similar to that for a function 
$\tilde{g}=\operatorname{ess\,sup} g - g$.
\end{proof}

For example, by \eqref{eq:eigenvalue}, a sufficient condition for
$\operatorname{ess\,inf} g$ not being an eigenvalue is that
\begin{equation*}
  \int_{\{\mat{x}\in\Omega\,:\,\operatorname{ess\,inf} g = g(\mat{x})\}} w(\mat{x})\,d\mat{x}=0
\end{equation*}
or that none of the basis functions is a characteristic function.
For another example, when $g(x)=x$ and the basis functions are polynomials,
we have the
well-known property of Gaussian quadrature that the endpoints are
not nodes of the Gaussian quadrature rule \cite[Theorem~1.46]{Gautschi:2004}.
For a Gaussian quadrature rule, 
this property follows from two facts: multiplication operator
$\operator{M}[g]$ does not have 
eigenvalues for $g(x)=x$,
and a finite-order polynomial cannot be a characteristic function. 

Once we have shown that the integrand of an improper integral cannot
have singularities at the eigenvalues of $\mat{M}_n[g]$, we
can use Jensen's operator inequality to show the
monotone growth of the approximations.
\begin{lemma}
  \label{lemma:operator-convex-monotone}
  Let $\operator{A}$ be a self-adjoint operator with an
  infinite matrix representation $\mat{A}_\infty$. Let $\mat{A}_\infty$
  and $\mat{v}\in\ell^2_0$ have finite approximations
  with elements
  $%
    [\mat{A}_n]_{i,j}=[\mat{A}_\infty]_{i,j}$ and $
    [\mat{v}_n]_i=[\mat{v}]_i
  $ %
  for $i,j=0,1,\ldots,n$.
  Let $f$ be an operator convex function on the interval $I$ that 
  contains
  the spectrum of $\mat{A}_\infty$.
  Then for $n$ more than the non-zero elements of $\mat{v}$, we have
  \begin{equation*}
    \mat{v}_n^*\,f(\mat{A}_n)\,\mat{v}_n\leq
    \mat{v}_{n+1}^*\,f(\mat{A}_{n+1})\,\mat{v}_{n+1}.
  \end{equation*}
\end{lemma}
\begin{proof}
  We define projection operators $\mat{P}_n$ as in \eqref{eq:projection}.
  First, we use Jensen's operator inequality \eqref{eq:jensen}
  with $s\in I$
  for operator
  \begin{equation*}
    \operator{B}=\mat{P}_{n+1}\,\mat{A}_\infty\,\mat{P}_{n+1}
    +s\,(\mat{I}_\infty-\mat{P}_{n+1})
    =
    \left[
      \begin{array}{cc}
        \mat{A}_{n+1}           & \mat{0}_{(n+2)\times\infty}\\
        \mat{0}_{\infty\times (n+2)} & s\,\mat{I}_\infty
      \end{array}
      \right].
  \end{equation*}
  We notice that $\operator{B}$ is a bounded operator in infinite dimensional
  Hilbert space and $\sigma(\operator{B})\subset I$. Therefore,
  Jensen's operator inequality \eqref{eq:jensen} applies.
  We also notice that
  \begin{align*}
    \left[
      \begin{array}{cc}
        \mat{A}_n             & \mat{0}_{(n+1)\times\infty}\\
        \mat{0}_{\infty\times (n+1)} & s\,\mat{I}_\infty \\
      \end{array}
      \right]
    &=
    \mat{P}_n
    \left[
      \begin{array}{cc}
        \mat{A}_{n+1}             & \mat{0}_{(n+2)\times\infty}\\
        \mat{0}_{\infty\times (n+2)} & s\,\mat{I}_\infty \\
      \end{array}
      \right]
    \,\mat{P}_n
    +s\,(\mat{I}_\infty-\mat{P}_n)\\
    &=
    \mat{P}_n\,\operator{B}\,\mat{P}_n
    +s\,(1-\mat{P}_n),
  \end{align*}
  and hence obtain
  \begin{align*}
    \left[
      \begin{array}{cc}
        f(\mat{A}_n)             & \mat{0}_{(n+1)\times\infty}\\
        \mat{0}_{\infty\times (n+1)} & \mat{0}_\infty \\
      \end{array}
      \right]
    &=
    \mat{P}_n\,
    \left[
      \begin{array}{cc}
        f(\mat{A}_n) & \mat{0}_{(n+1)\times\infty} \\
        \mat{0}_{\infty\times (n+1)} & f(s)\,\mat{I}_\infty \\
      \end{array}
      \right]
    \,\mat{P}_n\\
    &=
    \mat{P}_n\,
    f
    \left(
    \left[
      \begin{array}{cc}
        \mat{A}_n & \mat{0}_{(n+1)\times\infty} \\
        \mat{0}_{\infty\times (n+1)} & s\,\mat{I}_\infty \\
      \end{array}
      \right]
    \right)
    \,\mat{P}_n\\
    &=
    \mat{P}_n\,
    f(\mat{P}_n\,\operator{B}\,\mat{P}_n+
    s\,(\mat{I}_\infty-\mat{P}_n))\,\mat{P}_n\\
    & \qfleq
    \mat{P}_n\,f(\operator{B})\,\mat{P}_n\\
    &=
    \mat{P}_n\,
    \left[
      \begin{array}{cc}
        f(\mat{A}_{n+1})          & \mat{0}_{(n+2)\times\infty} \\
        \mat{0}_{\infty\times (n+2)} & f(s)\,\mat{I}_\infty \\
      \end{array}
      \right]
    \,\mat{P}_n\\
    &=
    \mat{P}_n\,
    \left[
      \begin{array}{cc}
        f(\mat{A}_{n+1})          & \mat{0}_{(n+2)\times\infty} \\
        \mat{0}_{\infty\times (n+2)} & \mat{0}_\infty \\
      \end{array}
      \right]
    \,\mat{P}_n,
  \end{align*}
  which implies 
  $\mat{v}_n^*\,f(\mat{A}_n)\,\mat{v}_n\leq
  \mat{v}_{n+1}^*\,f(\mat{A}_{n+1})\,\mat{v}_{n+1}$.
\end{proof}

With an unbounded version of Jensen's operator inequality
\eqref{eq:jensen},
we could also establish
$\mat{v}^*_n\,f(\mat{A}_n)\,\mat{v}_n\leq \mat{v}^*\,f(\mat{A}_\infty)\,\mat{v}$.
We do know that operator convex functions of positive bounded operators,
like $f(x)=x^{-1}$, are operator convex for positive unbounded operators
as well \cite[Theorem~4.3]{Dinh+Tikhonov+Veselova:2019}.

For further use of Lemma~\ref{lemma:operator-convex-monotone},
without an unbounded version of \eqref{eq:jensen},
we notice that function $f(x)=(x+s)^{-p}$ is operator
convex for any $p\in[0,1]$ and $s\geq 0$ because operator $\operator{B}+s$ is also a bounded
operator for any bounded operator $\operator{B}$ and $\sigma(\operator{B}+s)\in(0,\infty)$
if $\sigma(B)\in(0,\infty)$.
We need also some extension for the class of operator monotone functions.
\begin{lemma}
  Function $h(x)=-(x+s)^{-p}$ is operator monotone for positive unbounded self-adjoint operators for
  all $s\in (0,\infty)$ and $p\in[0,1]$.
\end{lemma}
\begin{proof}
  Function $h(x)=-(x+s)^{-p}$ is operator monotone on interval $[0,\infty)$ for
    any positive bounded operator $\operator{B}$
    because operator $s+\operator{B}$ is a positive bounded operator on interval
    $[s,\infty)\subset (0,\infty)$ and
      $\tilde{h}(x)=-x^{-p}$ is operator monotone on interval $(0,\infty)$.
      Function $h$ is then operator monotone also for all positive self-adjoint operators.

      For extending this result to the unbounded case, see the proof in
      \cite[Theorem~5]{Dinh+Tikhonov:2010}, which does not depend
      on the operator algebraic concepts of von Neuman algebra and affiliated
      operators mentioned in the statement of the theorem.
\end{proof}

This allows us to prove convergence for functions bounded by $x^{-p}$
for the following theorem.
\begin{theorem}
\label{thm:finite-improper-integral}
Let $p\in[0,1]$.
Let self-adjoint operator $\operator{A}$ have an infinite matrix representation
$\mat{A}_\infty$. Let $\mat{u}, \mat{v}\in\ell^2_0$ be such that
$\abs{\mat{u}^*$ $(\mat{A}_\infty - c\,\mat{I}_\infty)^{-p}\,\mat{v}} <\infty$.
Let function $f$ have a singularity at endpoint $c$ of the spectrum of
$\mat{A}_\infty$
such that $\abs{f(x)}\leq a + \frac{b}{\abs{x-c}^p}$ and $\mat{A}_\infty$ has no 
eigenvalue and eigenvector pair such that the endpoint $c$ is the eigenvalue
and the corresponding eigenvector is in $\ell^2_0$. 
Let finite
approximations be
$%
  [\mat{A}_n]_{i,j}=[\mat{A}_\infty]_{i,j}$,
  $[\mat{u}_n]_i = [\mat{u}]_i$,
  $[\mat{v}_n]_i = [\mat{v}]_i$
for $i,j=0,1,\ldots,n$. Then
$
  \mat{u}_n^*\,f(\mat{A}_n)\,\mat{v}_n\rightarrow
  \mat{u}^*\,f(\mat{A}_\infty)\,\mat{v}
  $
  as $n\rightarrow\infty$.
\end{theorem}
\begin{proof}
  It is enough to consider convergence for
  $  \mat{v}_n^*\,f(\mat{A}_n)\,\mat{v}_n\rightarrow
  \mat{v}^*\,f(\mat{A}_\infty)\,\mat{v}$ when $\mat{v}\in\ell^2_0$. 
  The more general convergence follows
  from the polar decomposition
  \begin{align*}
    4\,\mat{u}^*\,f(\mat{A}_\infty)\,\mat{v} 
    = &
    (\mat{u}+\mat{v})^*\,f(\mat{A}_\infty)\,(\mat{u}+\mat{v})
    -(\mat{u}-\mat{v})^*\,f(\mat{A}_\infty)\,(\mat{u}-\mat{v})\\
    &
    +i\,(\mat{u}+i\,\mat{v})^*\,f(\mat{A}_\infty)\,(\mat{u}+i\,\mat{v})
    -i\,(\mat{u}-i\,\mat{v})^*\,f(\mat{A}_\infty)\,(\mat{u}-i\,\mat{v}).
  \end{align*}
  If $\mat{u},\mat{v}\in\ell^2_0\cap\vectorspace{D}(\sqrt{\abs{f(\mat{A}_\infty)}})$, then also
  $\mat{u}+\mat{v},\mat{u}-\mat{v},\mat{u}+i\,\mat{v},
  \mat{u}-i\,\mat{v}\in\ell^2_0\cap\vectorspace{D}(\sqrt{\abs{f(\mat{A}_\infty)}})$.

  It is also enough to show convergence 
  when the lower endpoint of $\sigma(\mat{A}_\infty)$ is 0,
    and $\mat{A}_\infty$ has no eigenvalue at 0 such that the corresponding
    eigenvector is in $\ell^2_0$. Then we can select
    $h(x)=x^{-p}$, which is operator convex function on $(0,\infty)$
    and $-h(x)$ is operator monotone on $(0,\infty)$. Convergence
    for more general $\tilde{h}(x)=a+\frac{b}{\abs{x-c}^p}$ follows by simple
    transformations.

    We define a sequence of functions for $s> 0$ as
    \begin{equation*}
      a_n(s)=\mat{v}^*\,
      h\left(
      \left[
        \begin{array}{cc}
          \mat{A}_n & \mat{0}_{(n+1)\times\infty}\\
          \mat{0}_{\infty\times (n+1)} & \mat{0}_\infty
          \end{array}
        \right]+s
      \right)
      \,\mat{v}.
    \end{equation*}
    Functions $a_n(s)$ are bounded for all values of $n<\infty$ and
    $s> 0$. 
    By lemma \ref{lemma:operator-convex-monotone}, sequence $a_n(s)$ is 
    monotone increasing when $n$ is greater than the number of 
    non-zero elements in $\mat{v}$. Functions $a_n(s)$ are monotone decreasing
    as functions of $s$ because $-h(x+s)$ is operator monotone on
    $[0,\infty)$ also for unbounded positive self-adjoint infinite
      matrices and 
    \begin{equation*}
      \left[
        \begin{array}{cc}
          \mat{A}_n & \mat{0}_{(n+1)\times\infty}\\
          \mat{0}_{\infty\times (n+1)} & \mat{0}_\infty
        \end{array}
        \right]+s\qfleq
     \left[
       \begin{array}{cc}
         \mat{A}_n & \mat{0}_{(n+1)\times\infty}\\
         \mat{0}_{\infty\times (n+1)} & \mat{0}_\infty
       \end{array}
       \right]+r,
    \end{equation*}
    when $s\leq r$. We further define a double sequence $a_{n,m}=a_n(m^{-1})$ which is monotone increasing in
    both $n$ and $m$.
    By \cite[Theorem~4.2]{Habil:2006}, we can change the order of the limits for $n$ and $m$ and have
    \begin{align*}
      \lim_{n\rightarrow\infty}\mat{v}_n^*\,h(\mat{A}_n)\,\mat{v}_n
      &=
      \lim_{n\rightarrow\infty}\lim_{m\rightarrow \infty} a_{n,m}
      =
      \lim_{m\rightarrow \infty}\lim_{n\rightarrow\infty} a_{n,m}
      =
      \lim_{m\rightarrow \infty}\mat{v}^*\,h(\mat{A}_\infty+m^{-1})\,\mat{v}\\
      &=
      \mat{v}^*\,h(\mat{A}_\infty)\,\mat{v},
    \end{align*}
    because $h(\mat{A}_\infty+m^{-1})$ is a bounded operator.
    Convergence for $\abs{f(x)}\leq \tilde{h}(x)=a+\frac{b}{\abs{x-c}^p}$ follows in a similar
    way as the proof of theorem \ref{thm:dominated-convergence}.
\end{proof}

For an improper numerical integral on a finite endpoint,
we can present the following
theorem.
\begin{theorem}
  \label{thm:end-point-singularity}
  Let orthonormal functions $\phi_0=1,\phi_1,\phi_2,\ldots$ be dense
  in $\vectorspace{L}_w^2(\Omega)$. Let a function $g$ satisfy
  \eqref{eq:mult-mat}
  and let matrices $\mat{M}_n[g]$ have
  elements as in \eqref{eq:matrix-elements}.
  Let $c$ be such $\operatorname{ess\,sup} g$ or $\operatorname{ess\,inf} g$
  that it is not an eigenvalue of $\mat{M}_n[g]$, that is, $c$ is not an
  eigenvalue of $\operator{M}[g]$, or if it is, then the corresponding
  eigenfunction is not a finite linear combination of the basis functions
  $\phi_i$.
  Let a function $f$ be Riemann--Stieltjes integrable with respect to
  the spectral family of $\operator{M}[g]$ on any finite interval
  not containing $c$.
  Let $f$ be  bounded so that
  $\abs{f(x)}\leq a+\frac{b}{\abs{x-c}^p}$ for some $a,b\geq 0$ and $0\leq p\leq 1$
  that satisfy
  \begin{equation*}
    \int_\Omega
    \left(
    a+\frac{b}{\abs{g(\mat{x})-c}^p}
    \right)
    \,w(\mat{x})\,d\mat{x}<\infty.
  \end{equation*}
  Then
  \begin{equation*}
    \lim_{n\rightarrow\infty}[f(\mat{M}_n[g])]_{0,0}=
    \int_\Omega f(g(\mat{x}))\,w(\mat{x})\,d\mat{x}.
  \end{equation*}
\end{theorem}
\begin{proof}
This follows from Theorem~\ref{thm:finite-improper-integral}.
\end{proof}

\section{Numerical examples}
\label{sec:numerical-example}

As an example of convergence, we consider integration on the
interval $[0,\infty)$
  with exponential weight function $w(x)=e^{-x}$ and orthonormal
  polynomials (Laguerre polynomials) as the orthonormal basis functions.
  As an example, we use
  inside functions 
  $g_1(x)=\sqrt{x}$, $g_2(x)=x$, $g_3(x)=x^{3/2}$ and $g_4(x)=x^2$.
  
  The inside function $g_2$ corresponds to Gauss--Laguerre quadrature and
  has the well-known good convergence properties of Gaussian quadrature
  \cite{Uspensky:1928,Jouravsky:1928,Shohat+Tamarkin:1963,Freud:1971,Bultheel+Diaz-Mendoza+Gonzalez-Vera+Orive:2000}.
  The inside functions $g_1,g_2,g_3,g_4$ satisfy condition
  \eqref{eq:mult-mat}.
  Matrix
  $\mat{M}_\infty[g_2]$ is the tridiagonal Jacobi matrix for the Laguerre
  polynomials. Its elements are all nonnegative. We can see
  this by looking at the exact formula for the recurrence coefficients
  \cite[Table~1.1 and Definition~1.30]{Gautschi:2004} as well as in the first
  few matrix
  coefficients \eqref{eq:Laguerre-system-Jacobi-matrix}.
  Also, $\mat{M}_\infty[g_4]$ has nonnegative elements because it
  is a square of a matrix $\mat{M}_\infty[g_2]$ that has only nonnegative
  elements. On the contrary, matrices $\mat{M}_\infty[g_1]$ and 
  $\mat{M}_\infty[g_3]$ have both positive
  and negative elements, as we can see from the first few matrix coefficients
    \begin{align*}
      \mat{M}_2[g_1] &=
      \sqrt{\pi}\,
    \left[
      \begin{array}{cccc}
   \frac{1}{2} &       \frac{1}{4} &       -\frac{1}{16}\\
   \frac{1}{4} &   \frac{7}{8} &   \frac{11}{32} \\
   -\frac{1}{16} & \frac{11}{32} & \frac{145}{128}
      \end{array}
      \right],      \\
      \mat{M}_3[g_3] &=
      \sqrt{\pi}\,
    \left[
      \begin{array}{cccc}
        \frac{3}{4} & \frac{9}{8} & \frac{9}{32} & -\frac{3}{64}\\
        \frac{9}{8} & \frac{57}{16} & \frac{207}{64} & \frac{81}{128}\\
        \frac{9}{32} & \frac{207}{64} & \frac{1947}{256} & \frac{3051}{512}\\
        -\frac{3}{64} & \frac{81}{128} & \frac{3051}{512} & \frac{12873}{1024}
      \end{array}
      \right].
  \end{align*}

    We test convergence with two test functions $f_1(x)=\sin(\sqrt{x})$
    and $f_2(x)=\frac{\exp(x)}{1+x^2}$. 
    The first one is bounded, and the
    second one is a very fast-growing function. 

    For comparison, we formulate the numerical integral as
    \begin{equation}
      \label{eq:basic-example}
      \int_0^\infty f_i(x)\,w(x)\,dx\approx[f_i(g_j^{-1}(\mat{M}_n[g_j]))]_{0,0},
    \end{equation}
    where $g_j^{-1}$ exists for all $g_1,g_2,g_3,g_4$. This way, we get four 
    different quadrature rules for approximating the integral.
    We compute
    the finite matrix approximation of the multiplication operator symbolically
    as described in \cite[Remark~1]{Sarmavuori+Sarkka:2019}
    and the
    eigenvalue decomposition numerically with 64-bit IEEE 754 floating point
    numbers. The reason for using symbolic computation until the eigenvalue
    decomposition is that we also have the Jacobi matrix for the
    Gauss--Laguerre quadrature
    in closed-form. This allows us to compare quadratures up to
    way more than 30 nodes. With purely 64-bit floating point computation,
    the method would work for up to only 10 to 15 nodes. This is significantly
    improved up to 15 to 20 nodes
    if centralised moments are used instead of the raw ones.

    The results of the numerical approximation are presented in Figures
    \ref{fig:ex1-error}, \ref{fig:ex-2}, and \ref{fig:ex-3}.
    The example for
    a bounded function in Figure~\ref{fig:ex1-error} demonstrates convergence
    for all the inside functions as expected by 
    Theorem~\ref{thm:quadratic-convergence}.
\begin{figure}
\centering
    \begin{subfigure}[b]{0.485\textwidth}            
            \includegraphics[width=\textwidth]{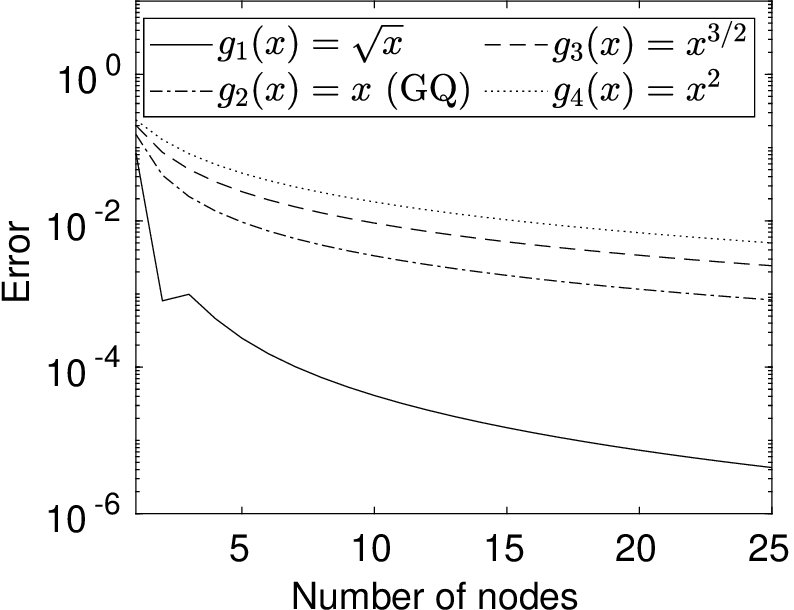}
            \caption{Approximation error.}
            \label{fig:ex1-error}
    \end{subfigure}%
    \quad
    \begin{subfigure}[b]{0.485\textwidth}
            \centering
            \includegraphics[width=\textwidth]{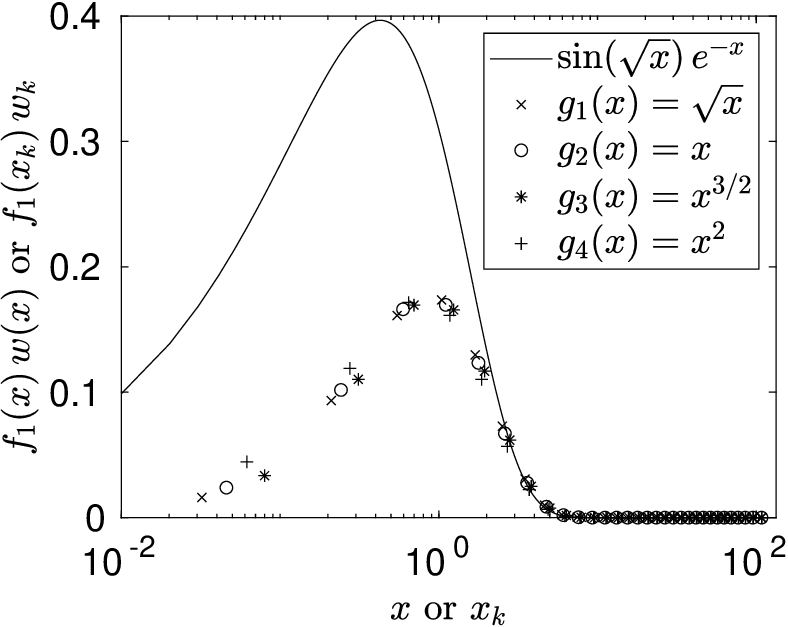}
            \caption{        Location of the nodes and weighted values of the
        function at the nodes.}
            \label{fig:ex1-nodes}
    \end{subfigure}
    \caption{
      Bounded outside function $f_1(x)=\sin(\sqrt{x})$ and the inside functions $g_1,g_2,g_3,g_4$.
    }\label{fig:ex-1}
\end{figure}

    Convergence is not very fast in Figure~\ref{fig:ex1-error}.
    For understanding the reason better, we can take a look at the
    location of the nodes and how much each node contributes to
    the numerical integral. In Figure~\ref{fig:ex1-nodes},
    we see each node located on the horizontal axis.
    On the vertical axis, we show the value $f_1(x_k)\,w_k$ which tells
    how much the term contributes to the total sum
    $\sum_{k=0}^{30} f_1(x_k)\,w_k$. For comparison, also the
    function graph is shown weighted by $w(x)$. The graph and the weights
    are not directly comparable so that it would be a
    better fit when the points $f_1(x_k)\,w_k$ are located
    on the graph.
    When nodes are further away from each other, the weight
    should be higher compared to a situation where nodes are close
    to each other.
    Notice also that the $x$-axis is on a logarithmic scale to show
    the interesting area of the function more clearly. This 
    further distorts the effect that values $f_1(x_k)\,w_k$ seem far off
    from the curve $f_1(x)\,w(x)$.
    However, the curve of $f_1(x)\,w(x)$ does show the area
    where the $f_1(x_k)\,w_k$ points should be higher assuming that the
    distance between the nodes does not vary very much.

    Figure~\ref{fig:ex1-nodes} shows that only 7 or 8 nodes are
    located in the area where $f_1(x)\,w(x)$ is high. The rest of the
    31 nodes are located in low area.
    This is quite natural because the support of the integral is
    $[0,\infty)$.
    Because the nodes have the interleaving
    property \cite[Theorem~2]{Sarmavuori+Sarkka:2019},
    increasing the number of nodes from 31 to 32 is going to add at most
    one more node in the interesting area. The majority of the nodes are going to
    have a very small contribution to the numerical integral.
    
    For the fast-growing outside function in Figure~\ref{fig:ex2-error},
    we have different results.
    For $g_1$ and $g_3$, we have divergence starting before 15 nodes. 
    This is in line
    with the fact that we do not have a proof of convergence for $g_1$ 
    and $g_3$, unlike
    for $g_2$ and $g_4$.
    The absence of proof, of course, does not mean that the numerical
    integral is automatically divergent. We do not have a proof that this
    is diverging either.
    In principle, it could be convergent in the
    end but just having the error growing very large for a large number
    of nodes until convergence would happen. This temporary
    large error might become too large for the floating point representation.

    The only theorem that we have been able to prove which applies to
    $g_1$ and $g_3$ is Theorem~\ref{thm:quadratic-integral-convergence},
    which says that for $g_1(x)=\sqrt{x}$, $f$ would have to be linearly
    bounded,
    and for $g_3(x)=x^{\frac{3}{2}}$, the bound would have to be
    $\abs{f(x)}\leq a+x^3$ for some $a>0$.
    The true bound where convergence starts to fail is probably
    somewhere between these quite slowly growing functions and
    the example function that is almost too fast-growing function to be
    integrable at all.

    Another possible explanation for the divergence
    could be numerical instability. In order to test this hypothesis,
    we computed the eigenvalue decomposition and the numerical integral
    with variable precision for
    $g_1$ on 11 nodes where the divergence has already started, and
    the approximation error is $0.2323$. The difference between the
    floating point and the variable precision calculation is
    $1.7170\cdot 10^{-16}$ for both 32- and 64-digit variable precision.

    Convergence for $g_2$ follows from the well-known
    results
    \cite{Uspensky:1928,Jouravsky:1928,Shohat+Tamarkin:1963,Freud:1971,Bultheel+Diaz-Mendoza+Gonzalez-Vera+Orive:2000}. 
    Convergence for
    $g_4$ follows from Theorem~\ref{thm:power-series-convergence} because
    \begin{align*}
      f_2(g_4^{-1}(x))
      &=
      \frac{e^{\sqrt{x}}}{1+x}
      \leq 
      \frac{e^{\sqrt{1+x}}+e^{-\sqrt{1+x}}}{1+x}
      =
      2\,\frac{\cosh(\sqrt{1+x})}{1+x}\\
      &=
      \frac{2}{1+x}+
      2\sum_{k=0}^\infty \frac{1}{(2(k+1))!}\,(1+x)^k.
      \end{align*}    
\begin{figure}
\centering
    \begin{subfigure}[b]{0.485\textwidth}            
            \includegraphics[width=\textwidth]{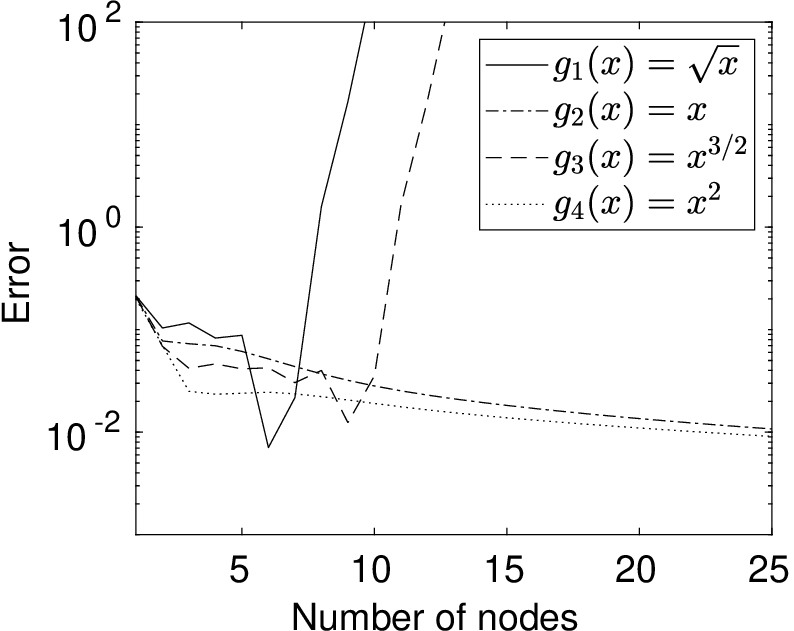}
            \caption{Weighting function $h_1=1$}
            \label{fig:ex2-error}
    \end{subfigure}%
    \quad
    \begin{subfigure}[b]{0.485\textwidth}
            \centering
            \includegraphics[width=\textwidth]{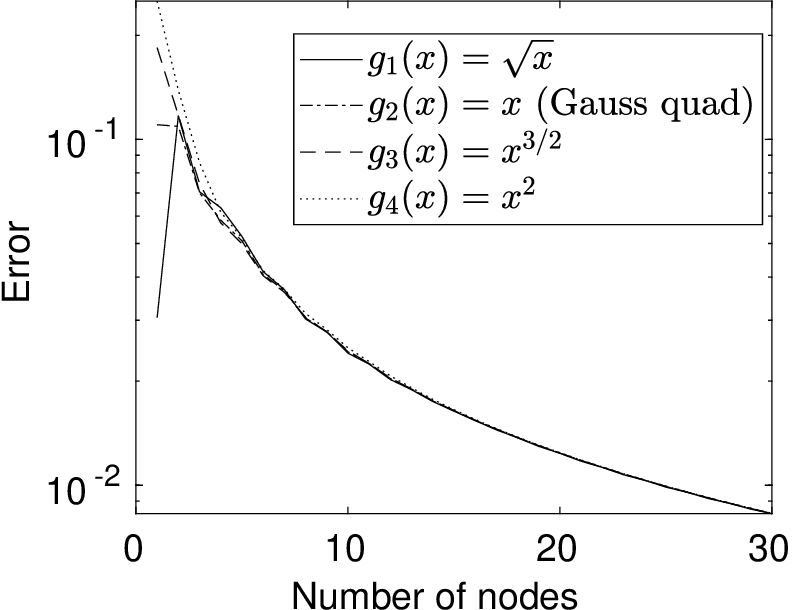}
            \caption{Weighting function 
              $h_2(x)=\frac{\exp(x/2)}{(1+x^2)^{3/8}}$
            }
            \label{fig:ex3-error}
    \end{subfigure}
    \caption{
Approximation error for fast-growing outside function 
    $f_2(x)=\frac{\exp(x)}{1+x^2}$ and
inside functions $g_1,g_2,g_3,g_4$ with weighting function $h_1=1$ and
$h_2(x)=\frac{\exp(x/2)}{(1+x^2)^{3/8}}$.
    }\label{fig:ex-2}
\end{figure}

    We can make a quadrature rule based on the inside function $g_1$ or $g_3$
    convergent for the outside function $f_2$ by
    Theorem~\ref{thm:reweighted-convergence}. We select a function
    $h$ of Theorem~\ref{thm:reweighted-convergence} as
    $h_2(x)=\frac{\exp(x/2)}{(1+x^2)^{3/8}}$. The function
    $r(x)=\frac{f_2(x)}{h_2^2(x)}$ is bounded, and we can
    formulate the integral as
    \begin{equation*}
      \int_0^\infty f_2(x)\,w(x)\,dx=
      \mat{v}^\top\,r(g_j^{-1}(\mat{M}_\infty[g_j]))\,\mat{v}
      \approx
      \mat{v}_n^\top\,r(g_j^{-1}(\mat{M}_n[g_j]))\,\mat{v}_n,
    \end{equation*}
    where $\mat{v}\in\ell^2$ has components
    $%
      [\mat{v}]_i=\langle h_2,\phi_i\rangle=
      \int_0^\infty h_2(x)\,\phi_i(x)\,w(x)\,dx
    $, %
    and $\mat{v}_n$ is a truncation of $\mat{v}$ up to component $n$.
    The approximation in \eqref{eq:basic-example} is also a special case
    of this more general approach with function $h$ selected as $h_1=1$.
    Functions $h_1$ and $h_2$ do not affect the placement of the
    nodes but only the weights of the numerical integration as
    given in \eqref{eq:general-weights}.
    Figure~\ref{fig:ex3-error} shows that with weight function $h_2$,
    this approach is convergent
    also for the inside functions $g_1$ and $g_3$.

    We can again look at the location of the nodes and corresponding
    values of $f_2(x_k)\,w(x_k)$ in Figure~\ref{fig:ex3-nodes},
    which shows that the weights for large values of $x_k$ are
    very high for $g_1$ and $g_3$. The location of the nodes is not
    very different for all of the quadratures.

    \begin{figure}[htbp]
      \centering
      \includegraphics[scale=0.71]{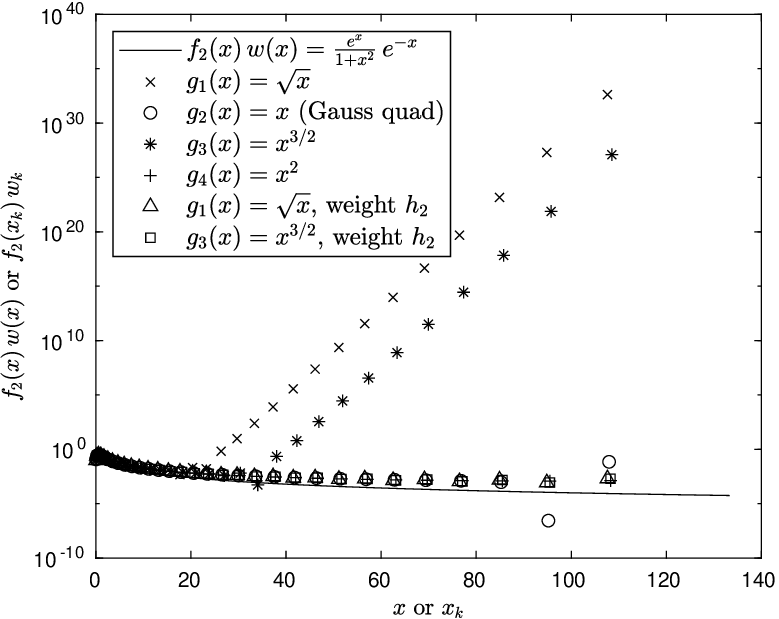}
      \caption{
        \label{fig:ex3-nodes}
        Location of the nodes and weighted value of the
        function for them.
      }
    \end{figure}

    Finally, we note that using $h_2$ instead of $h_1$ does not universally 
    improve
    convergence as shown in Figure~\ref{fig:ex4-error}. The results for
    inside functions $g_3,g_4$ are similar to that for $g_2$.
\begin{figure}
\centering
    \begin{subfigure}[b]{0.485\textwidth}            
            \includegraphics[width=\textwidth]{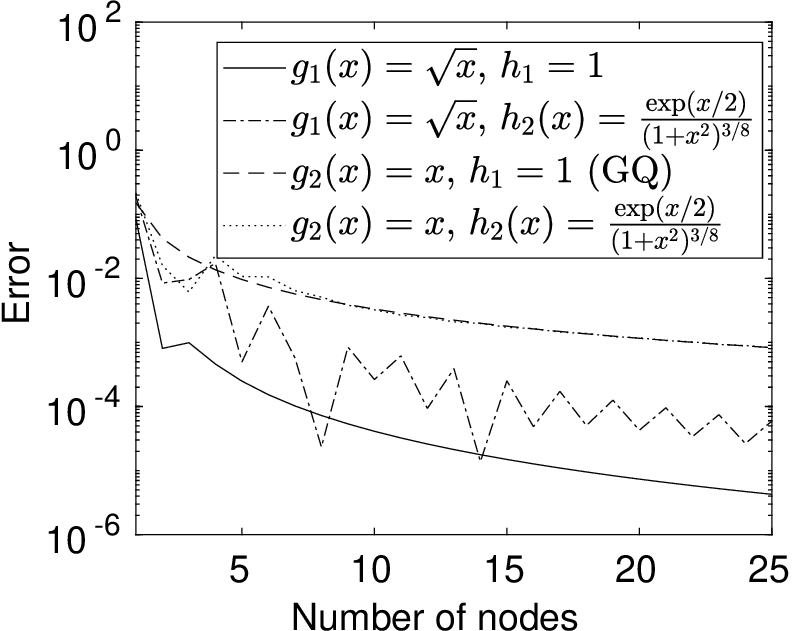}
            \caption{Approximation error for the bounded outside function
        $f_1$ and
        the inside functions $g_1,g_2$ with weighting functions 
        $h_1,h_2$.}
            \label{fig:ex4-error}
    \end{subfigure}%
    \quad
    \begin{subfigure}[b]{0.485\textwidth}
            \centering
            \includegraphics[width=\textwidth]{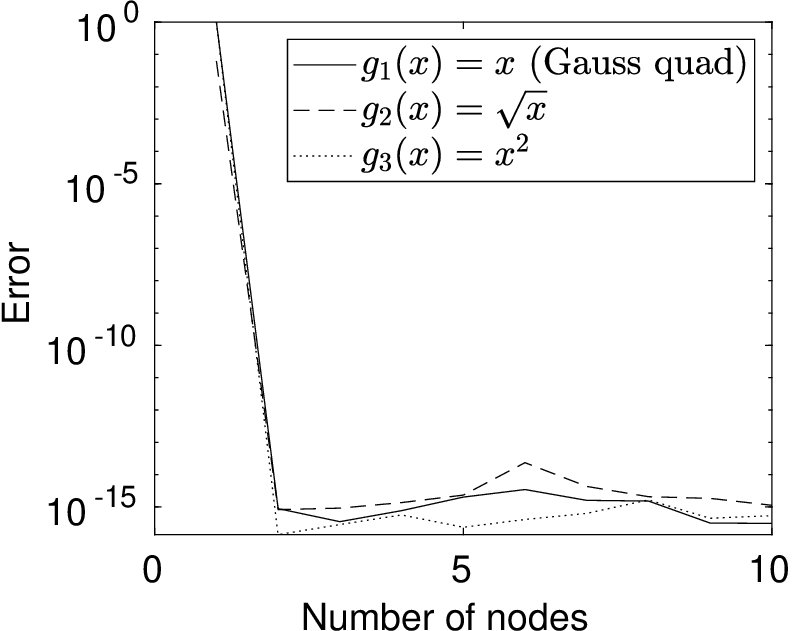}
            \caption{Approximation error for Thomae's function
              \eqref{eq:thomaes}
              using quadratures based on the inner functions $g_1,g_2,g_3$.
            }
            \label{fig:ex-thomaes}
    \end{subfigure}
    \caption{
      Approximation error for bounded outside functions $f_1(x)=\sin(\sqrt{x})$
      and Thomae's function \eqref{eq:thomaes}.
    }\label{fig:ex-3}
\end{figure}

    As an example of a discontinuous function, we use
    Thomae's function
    \begin{equation}
      \label{eq:thomaes}
      f(x)=
      \left\{
      \begin{array}{ll}
        \frac{1}{q} & \mathrm{if}~x=\frac{p}{q}~\,\mathrm{for~coprime}~p,q\in\mathbb{Z}^+,\\
        0 & \mathrm{otherwise}.
      \end{array}
      \right.
    \end{equation}
    The true value of the integral for Thomae's function
    is 0 \cite[Example~7.1.6]{Bartle+Sherbert:2011}.
    The results are shown in Figure~\ref{fig:ex-thomaes}.
    Convergence is very fast, although the function value
    at every node is nonzero. The true value of the node could be
    an irrational number, which would make convergence even faster.
    However, we define the node to be an approximation
    to the eigenvalue in floating point representation which is always
    a rational number. Even in that case,
    the nonzero values are so
    small that the integral quickly converges close to 0.
    
    As an example of a multidimensional integral, we consider
    the following integral
    \begin{equation}
      \label{eq:nd-example}
      I=\int_{\mathbb{R}^3} \frac{1}{\|\mat{x}\|^2}\,w(\mat{x})\,d\mat{x},
    \end{equation}
    where the weight function is
    Gaussian
    $w(\mat{x})=\pi^{-\frac{3}{2}}\,e^{-\|\mat{x}\|^2}$.
    We select functions $f$ and $g$ in  \eqref{eq:integral} as
    $f(x)=\frac{1}{x}$ and $g(\mat{x})=\|\mat{x}\|^2$.
    This turns the integration problem into
    a linear algebra problem. The approximation for \eqref{eq:nd-example}
    is given by $I\approx\mat{e}_0^\top\,\mat{M}_n[g]^{-1}\,\mat{e}_0$.
    Function $f$ is singular at 0 and the whole integrand
    at $\mat{0}$. We expect convergence based on
    Theorem~\ref{thm:end-point-singularity}, and convergence
    should be monotone based on Lemma~\ref{lemma:operator-convex-monotone}.

    We get the matrix elements for the function of the multi-dimensional
    variable from the elements of a matrix for a one-dimensional variable.  
    We start with the Jacobi matrix for the Hermite system, that is, the
    orthogonal system with weight function $w(x)=e^{-x^2}$.

    The Jacobi matrix is symmetric and tridiagonal.
    We can write a general symmetric tridiagonal infinite matrix as
    \begin{equation}
      \label{eq:tridiagonal-matrix}
      \left[
        \begin{array}{ccccccc}
          u_0 & v_0     &        &        &        &       & 0 \\
          v_0 & u_1     & v_1    &        &        &       &    \\
              & \ddots & \ddots & \ddots &        &        &   \\
              &        & v_{k-1} & u_k    & v_k     &        &   \\
              &        &        & \ddots & \ddots  & \ddots &   \\
          0   &        &        &        &         &       &
        \end{array}
        \right].
    \end{equation}
    For the Hermite system, the Jacobi matrix coefficients
    in \eqref{eq:tridiagonal-matrix} are given by $u_k=0$ and
    $v_k=\sqrt{\frac{1}{2}\,(k+1)}$ for $k=0,1,\ldots$
    \cite[Table~1.1 and Definition~1.30]{Gautschi:2004}.
    Thus, these are the coefficients of infinite matrix
    $\mat{M}_\infty[\operatorname{id}]$ for function $\operatorname{id}(x)=x$ in one dimension.

    A second power of $\mat{M}_\infty[\operatorname{id}]$ is
    a 5-diagonal infinite matrix with 2 side diagonals
    \begin{equation*}
      \label{eq:5-diagonal-matrix}
      \mat{M}_\infty[\operatorname{id}]^2=
      \left[
        \begin{array}{cccccccccc}
          a_0 & b_0     & c_0    &        &        &        &         &       & 0 \\
          b_0 & a_1     & b_1    & c_1    &        &        &         &        & \\
          c_0 & b_1     & a_2    & b_2    & c_2    &        &         &        & \\
              & \ddots & \ddots & \ddots & \ddots & \ddots &        &        & \\
              &        & c_{k-2} & b_{k-1} & a_k    & b_k     & c_k     &        & \\
              &        &        & \ddots & \ddots & \ddots & \ddots & \ddots & \\
          0   &        &        &        &        &        &        &        &
        \end{array}
        \right]
      =
      \left[
        \begin{array}{ccccccc}
          0 & v_0     &        &        &        &       & 0 \\
          v_0 & 0     & v_1    &        &        &       &    \\
              & \ddots & \ddots & \ddots &        &        &   \\
              &        & v_{k-1} & 0    & v_k     &        &   \\
              &        &        & \ddots & \ddots  & \ddots &   \\
          0   &        &        &        &         &       &
        \end{array}
        \right]^2,
    \end{equation*}
    where
      $a_k = v_{k-1}^2+v_k^2=k+\frac{1}{2}$,
      $b_k = 0$, and
      $c_k = v_k\,v_{k+1}=\frac{1}{2}\,\sqrt{(k+1)\,(k+2)}$.

    In the following notation, we use multi-indices as matrix indices.
    In computation, we use these multi-indices in graded lexicographic order \cite[Chapter~3.1]{Dunkl+Xu:2014}.
    The monomials of a multidimensional variable are $\mat{x}^{\mat{\alpha}}=\prod_{i=0}^{d-1} x_i^{\alpha_i}$ for $\mat{x}\in\mathbb{R}^d$ and
    $\mat{\alpha}\in\mathbb{N}_0^d$. Similarly, we can index the orthonormal polynomials with multi-indices as
    $\phi_{\mat{\alpha}}(\mat{x}) =\prod_{i=0}^{d-1}\phi_{\alpha_i}(x_i)$, where $\phi_i$ are the one-dimensional orthonormal polynomials.
    For matrix elements, the definition is
    $\mat{M}_\infty[g]_{\mat{\alpha},\mat{\beta}}=\langle \phi_{\mat{\alpha}},g\,\phi_{\mat{\beta}}\rangle$.

    For $g(\mat{x})=\|\mat{x}\|^2$, we get a matrix which has one main diagonal of elements on
    $(\mat{\alpha},\mat{\alpha})$ and several side diagonal like sets
    of elements on $(\mat{\alpha},\mat{\alpha}+2\,\mat{e}_i)$ for
    $i=0,1,2$. These sets are symmetric, that is,
    $\mat{M}_\infty[g]_{\mat{\alpha},\mat{\alpha}+2\,\mat{e}_i}=\mat{M}_\infty[g]_{\mat{\alpha}+2\,\mat{e}_i,\mat{\alpha}}$.
    The diagonal matrix elements are
    \begin{align*}
      \mat{M}_\infty[g]_{\mat{\alpha},\mat{\alpha}}
      &=\frac{3}{2}+\sum_{i=0}^{2}\alpha_i.
    \end{align*}
    The side diagonal elements are given for $i=0,1,2$ as
    \begin{align*}
      \mat{M}_\infty[g]_{\mat{\alpha},\mat{\alpha}+2\,\mat{e_i}}
      &=
      \mat{M}_\infty[g]_{\mat{\alpha}+2\,\mat{e_i},\mat{\alpha}}=\frac{1}{2}\,\sqrt{(\alpha_i+1)\,(\alpha_i+2)}.
    \end{align*}
    Otherwise, the elements are 0.

    For numerical comparison, we compute the numerical integral
    $\mat{e}_0^\top\,\mat{M}_n[g]^{-1}\,\mat{e}_0$ with different
    values of the total order of the polynomials
    $\abs{\mat{\alpha}}=\sum_{i=0}^{d-1}\alpha_i=1,2,\ldots,80$.
    The size of the matrix $\mat{M}_n[g]$ depends on
    the total order as $\binom{\abs{\mat{\alpha}}+3}{3}$
    \cite[Chapter~3.1]{Dunkl+Xu:2014}.
    At maximum, we make the computation for a matrix of size
    $\binom{80+3}{3}=91881$.
    We can compare this to the true value of \eqref{eq:nd-example}
    that we can find
    in the spherical coordinates as
      $I=\frac{4}{\sqrt{\pi}}\,\int_0^\infty e^{-r^2}\,dr=2$.
    Figure~\ref{fig:ex5} shows the values of the approximation
    for different values of the order of the polynomial approximation
    as well as the true value.
\begin{figure}
\centering
    \begin{subfigure}[b]{0.485\textwidth}            
            \includegraphics[width=\textwidth]{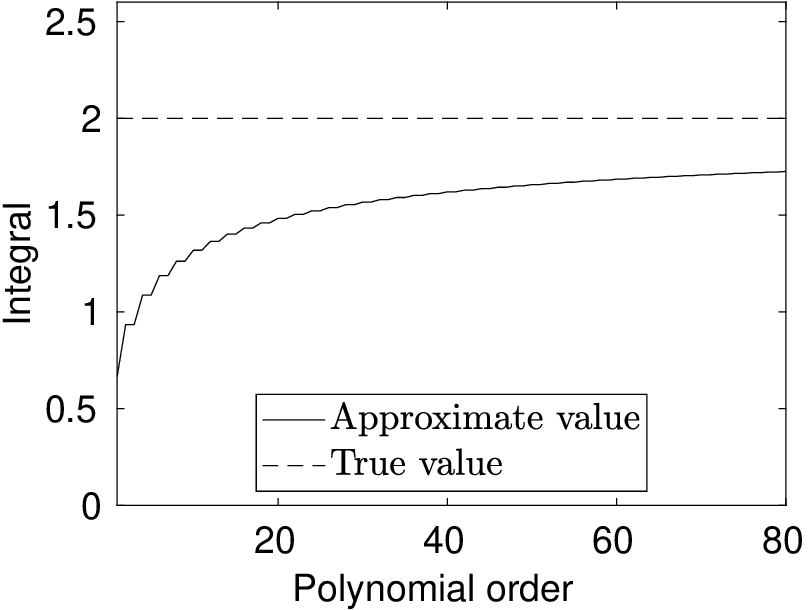}
            \caption{Numerical approximation for integral $I$ of \eqref{eq:nd-example}.
            }
            \label{fig:ex5}
    \end{subfigure}%
    \quad
    \begin{subfigure}[b]{0.485\textwidth}
            \centering
            \includegraphics[width=\textwidth]{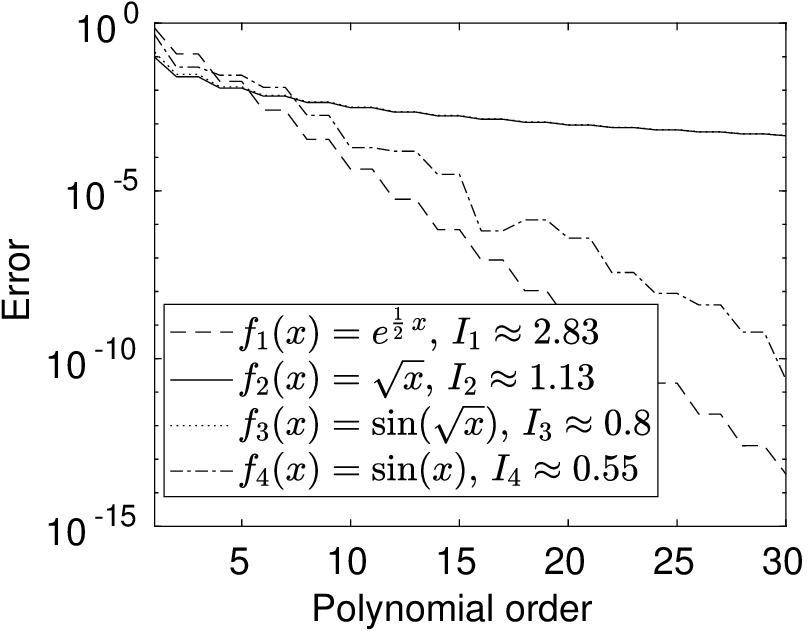}
            \caption{Approximation error of functions
        $f_k$ for integral in \eqref{eq:nd-fs}.
            }
            \label{fig:ex6}
    \end{subfigure}
    \caption{
      Results of numerical integration with inside function
      $g(\mat{x})=\|\mat{x}\|^2$.
    }%
\end{figure}
    We see that convergence is monotone as expected but very slow.

    The nodes, which we get from $\mat{M}_n[g]$, are not very efficient
    because of multiple eigenvalues. Moreover, the weights of the
    nodes corresponding to the multiple eigenvalues are 0. For example,
    in Table~\ref{tbl:eigenvalues}, we list the eigenvalues
    and weights for $\mat{M}_{55}[g]$ in a case where
    $\abs{\mat{\alpha}}\leq 5$.
    The size of the matrix is
    $\binom{5+3}{3}=56$.
    \begin{table}[htbp]
      \caption{Eigenvalues $\lambda_i$, the multiplicity of the eigenvalues
        $\mu(\lambda_i)$, and the corresponding weights $w_i$ for $\mat{M}_{55}[g]$ for function
        $g(\mat{x})=\|\mat{x}\|^2$ 
        and $w(\mat{x})=\pi^{-\frac{3}{2}}\,e^{-\|\mat{x}\|^2}$ in $\mathbb{R}^3$.}
      \label{tbl:eigenvalues}
      \begin{tabular}{|c|c|c|c|c|c|c|c|c|c|c|c|c|}
        \hline
        $\lambda_i$      & $0.6663$ & $1.2$ & $2.4$ & $2.8008$ & $3.2$ & $3.8$ & $5.5$ & $6.5$ & $6.6$ & $7.0329$ & $7.8$ & $8.5$ \\
        $\mu(\lambda_i)$ & 1        & 3      & 5      & 1      & 7     & 3     & 9     & 11    & 5     & 1        & 7     & 3 \\
        $w_i$            & 0.6400   & 0      & 0      & 0.3446 & 0     & 0     & 0     & 0     & 0     & 0.0154   & 0     & 0 \\
        \hline
      \end{tabular}
    \end{table}
    We see that only eigenvalues with multiplicity 1 have non-zero weight
    and there are only 3 of them in the total 56 eigenvalues.

    Figure~\ref{fig:nd-nodes} shows
    11 non-zero weights and the corresponding nodes for
    $\abs{\mat{\alpha}}\leq 20$, where the size of the
    matrix is 1771.
    \begin{figure}[htbp]
      \centering
      \includegraphics[scale=0.71]{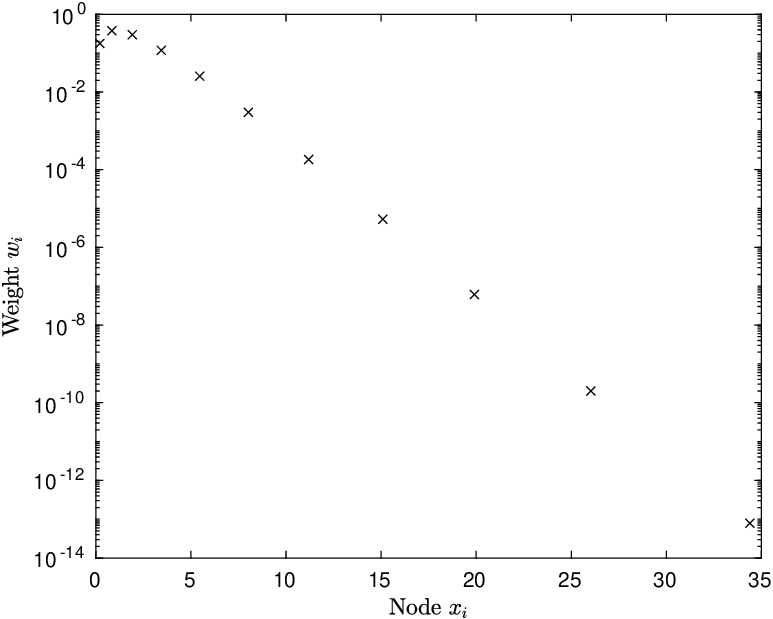}
      \caption{
        \label{fig:nd-nodes}
        Nodes and weights of 11-point quadrature rule computed
        from $\mat{M}_{1770}[g]$ for
        $\abs{\mat{\alpha}}\leq 20$.
      }
    \end{figure}
    We see that some of the weights are very small.
    However, these small weights are still relevant for
    fast-growing functions like $f_1(x)=e^{\frac{1}{2}\,x^2}$,
    as is shown in Figure~\ref{fig:ex6}. With the
    11-point quadrature in Figure~\ref{fig:nd-nodes}
    corresponding to polynomial order of 20,
    for $f_1$, 
    the 2 biggest contributions come from the 2 nodes
    with the 2 smallest weights.

    We further test the convergence
    of the same $\mat{M}_n[g]$
    for 4 different outside functions, that is,
    for integral
    \begin{equation}
      \label{eq:nd-fs}
      I_k=\int_{\mathbb{R}^3}f_k(g(\mat{x}))\,w(\mat{x})\,d\mat{x},
    \end{equation}
    where $f_1(x)=e^{\frac{1}{2}\,x^2}$, $f_2(x)=\sqrt{x}$,
    $f_3(x)=\sin(\sqrt{x})$, $f_4(x)=\sin(x)$,
    $g(\mat{x})=\|\mat{x}\|^2$, and $w$ is as in
    \eqref{eq:nd-example}.
    Again, we can find the true values of the integrals in the
    spherical coordinates as
    \begin{align*}
      I_1 &= \frac{4}{\sqrt{\pi}}\,
      \int_0^\infty r^2\, e^{-\frac{1}{2}\,r^2}\,dr
      =2\,\sqrt{2}\approx 2.828427,\\
      I_2 &= \frac{4}{\sqrt{\pi}}\,\int_0^\infty r^3\,e^{-r^2}\,dr
      =\frac{2}{\sqrt{\pi}}\approx 1.128379,\\
      I_3 &= \frac{4}{\sqrt{\pi}}\,\operatorname{Im}\left(
      \int_0^\infty r^2\,e^{r\,i-r^2}\,dr
      \right)
      =\frac{1}{\sqrt{\pi}}-\frac{i}{2}\,\operatorname{erf}
      \left(
      \frac{i}{2}
      \right)\,e^{-\frac{1}{4}}
      \approx 0.803652,\\
      I_4 &= \frac{4}{\sqrt{\pi}}\,\operatorname{Im}\left(
      \int_0^\infty r^2\,e^{(i-1)\,r^2}dr
      \right)
      =\frac{\sqrt[4]{2}\,\sqrt{2+\sqrt{2}}}{4}
      \approx 0.549342,
    \end{align*}
    where $\operatorname{Im}(x+y\,i)=y$.
    Figure~\ref{fig:ex6} shows some rate of convergence
    for all outside functions.
    Convergence follows
    from Theorem~\ref{thm:power-series-convergence}
    because $\mat{M}_\infty[g]$ has only nonnegative elements, and
    all the functions are bounded by
    $f_1(x)=e^{\frac{1}{2}\,x^2}=\sum_{k=0}^\infty \frac{1}{2^k\,k!}\,x^{2\,k}$.

\section{Conclusions}
\label{sec:conclusions}
We have proved several theorems on the convergence of the finite matrix
approximations of multiplication operators
for numerical integration. The theorems also apply generally to
self-adjoint operators.

For integration on an unbounded interval, we have shown several theorems
for strong resolvent convergence.
We  have proved strong resolvent convergence of finite matrix truncations
of self-adjoint operators with infinite matrix representation.
We have proved that the strong resolvent convergence for a sequence of self-adjoint
operators implies strong convergence for Riemann--Stieltjes integrable functions
and strong resolvent convergence for unbounded functions of
the operators.
We have proved that the strong resolvent convergence for a sequence of
self-adjoint operators implies the convergence of quadratically bounded functions
of the operators for certain quadratic forms that can be used in numerical
integration.
We have proved that for multiplication operators of polynomials
or self-adjoint operators with nonnegative coefficients, convergence
results can be extended to higher-order functions than
quadratically bounded functions.
We have also numerically demonstrated how convergence can fail
for very fast-growing functions if some matrix coefficients are negative.
We have shown a way to improve convergence by using weights that
are derived from a bounding function of the integrand.

We have also proved theorems for improper integrals with finite
endpoints.
We have shown that endpoints of integration are not eigenvalues
of the finite matrix approximation of the multiplication operator and,
therefore, not nodes of the matrix method of integration which is also a
property
of Gaussian quadrature. We have shown that numerical integrals converge
for improper integrals if the integrand function is bounded by an
operator convex function.

\backmatter

\bmhead{Acknowledgements}

We thank the anonymous reviewers for their valuable comments.
The work was supported by Academy of Finland.

\noindent

\begin{appendices}

\section{Riemann--Stieltjes and Darboux--Stieltjes Integrable Functions}
\label{app:integrable-functions}
The Riemann--Stieltjes integral has been defined in literature in two slightly
different ways. We shall adopt a practice of using two different names for the
two different definitions. First, we give the definition of the 
Riemann--Stieltjes 
integral and then the definition of the Darboux--Stieltjes integral.

A Riemann--Stieltjes integral over an interval $[a,b]$ of a function $f$ 
with respect to a monotonically non-decreasing function $\rho$ is the following
\cite{Kestelman:1960,Davis+Rabinowitz:1984}:
Let a partition of the interval $[a,b]$ be
$a=t_0<t_1<t_2<\ldots<t_n=b$. The norm of the partition is 
$%
\Delta=\max_k t_k-t_{k-1}.
$ %
 The Riemann--Stieltjes sum is
$%
S_n=\sum_{k=1}^n f(\xi_k)\,(\rho(t_k)-\rho(t_{k-1})).
$ %
If the sequence $S_n$ converges for all partitions for which
$\Delta\rightarrow 0$
and for any $\xi_k\in[t_{k-1},t_k]$,
then the function $f$ is Riemann--Stieltjes integrable with respect to
the function $\rho$.

The Riemann--Stieltjes integral is often defined so that the partitions are
restricted to be refined \cite[Definition~7.1]{Apostol:1981}. That kind    
of change in the definition is equivalent \cite[Theorem~7.19]{Apostol:1981} 
to our
 second definition, the Darboux--Stieltjes integral 
\cite[p.~250]{Kestelman:1960},
which is also sometimes termed Riemann--Stieltjes integral
\cite[Definition~6.2]{Rudin:1976}. A Darboux--Stieltjes integral over
an interval $[a,b]$ of a function $f$ with respect to
monotonically non-decreasing
function $\rho$ is the following:
Let a partition of the interval $[a,b]$ be
$a=t_0<t_1<t_2<\ldots<t_n=b$. We define
\begin{align*}
  M_k &= \sup_{t\in[t_{k-1}, t_k]} f(t),~
  m_k = \inf_{t\in[t_{k-1}, t_k]} f(t),~
  \Delta \rho_k = \rho(t_k)-\rho(t_{k-1}),\\
  U &= \sum_{k=1}^n M_k\,\Delta\rho_k,~
  L = \sum_{k=1}^n m_k\,\Delta\rho_k.
\end{align*}
Function $f$ is Darboux--Stieltjes integrable with respect to function $\rho$,
if
$%
  \inf U = \sup L
$ %
when $\inf$ and $\sup$ are taken over all partitions.

The Darboux--Stieltjes integral is more general than the Riemann--Stieltjes integral
\cite[p.~250]{Kestelman:1960}.
If $f$ is Riemann--Stieltjes integrable with respect to $\rho$, then
$f$ and $\rho$ are nowhere simultaneously discontinuous \cite[Theorem~317]{Kestelman:1960}. Darboux--Stieltjes integral does not have this limitation.
If $f$ is continuous whenever $\rho$ is
discontinuous, then the two definitions are equal \cite[p.~251]{Kestelman:1960}. 
For the Riemann- and most of the Darboux--Stieltjes integrable functions, we have
the following lemma.
\begin{lemma}
  \label{lemma:Darboux--Stieltjes}
  Let $\rho(t)$ be a right continuous bounded non-decreasing function on
  interval $[a,b]$. Let f be a Riemann--Stieltjes integrable function with
  respect to $\rho$ (or Darboux--Stieltjes integrable and continuous on the
  discontinuities of $\rho$). Then for every $\epsilon>0$, there is a partition
  $a=t_0<t_1<\ldots<t_n=b$ such that we can define step functions
  \begin{equation*}
    f_u(t)=
    \left\{
      \begin{array}{lllll}
        \sup\limits_{s\in[t_{k-1},t_k]} f(s), & \mathrm{when} & t\in[t_{k-1},t_k] & \mathrm{for} & k=1,2,\ldots,m,\\
        \sup\limits_{s\in[t_{n-1},t_n]} f(s), & \mathrm{when} & t=b              &              &
      \end{array}
      \right.
  \end{equation*}
  and
  \begin{equation*}
    f_l(t)=
    \left\{
      \begin{array}{lllll}
        \inf\limits_{s\in[t_{k-1},t_k]} f(s), & \mathrm{when} & t\in[t_{k-1},t_k] & \mathrm{for} & k=1,2,\ldots,m,\\
        \inf\limits_{s\in[t_{n-1},t_n]} f(s), & \mathrm{when} & t=b              &              &
      \end{array}
      \right.
  \end{equation*}
  so that
  \begin{equation*}
    \sum_{k=1}^n (f_u(t_{k-1})-f_l(t_{k-1}))\,(\rho(t_k)-\rho(t_{k-1}))<\epsilon.
  \end{equation*}
  The points $t_k$ can be chosen so that they are not discontinuity points of $\rho$.
\end{lemma}
\begin{proof}
  The lemma is a straightforward consequence of a well-known result
  \cite[Theorem~6.6]{Rudin:1976} except for the last sentence. The last sentence
  follows because, due to the continuity, there is a $\delta$ such that it is
  possible to move all the points $t_k$ right by an arbitrary value on
  the interval
  $[0,\delta]$. It is also possible to choose the new points on the intervals
  $[t_k,t_k+\delta]$ so that they are not discontinuities of $\rho$.

  We assume that the theorem holds without the last sentence for points $\tilde{t}_k$
  where $\rho$ may possibly be discontinuous. Let functions $\tilde{f}_u,\tilde{f}_l$ be
  determined like $f_u,f_l$ in the statement of the theorem,  except for the points
  $\tilde{t}_k$ instead of points $t_k$. By the theorem without the last sentence,
  we have
  \begin{equation}
    \label{eq:discontinuous-point-limit}
    \sum_{k=1}^n (\tilde{f}_u(\tilde{t}_{k-1})-\tilde{f}_l(\tilde{t}_{k-1}))\,
    (\rho(\tilde{t}_k)-\rho(\tilde{t}_{k-1}))<\frac{\epsilon}{2}.
  \end{equation}
  If $\rho$ is continuous at $\tilde{t}_k$, we set $t_k=\tilde{t}_k$. If not,
  we use the
  right continuity of $\rho$ and $f$ at point $\tilde{t}_k$. We select
  \begin{equation*}
    \nu=\frac{\epsilon}{4\,n\,(\sup \abs{f} + \sup \abs{\rho})}.
  \end{equation*}
  Due to the right continuity of $f$ and $\rho$, there is a $\delta>0$ so that
  for all $t_k$
  \begin{equation*}
    0 < t_k-\tilde{t}_k<\delta\Rightarrow
    \left\{
    \begin{array}{ll}
      \abs{f(t_k)-f(\tilde{t}_k)} & < \frac{\nu}{2},\\
      \rho(t_k)-\rho(\tilde{t}_k) & < \frac{\nu}{2}.
    \end{array}
    \right.
  \end{equation*}
  It is also important to select the new $t_k$ so that it is not a discontinuity
  point of $\rho$. This is possible because $\rho$ is a non-decreasing function
  and such a function has at most countable number of discontinuities
  \cite[Theorem~4.30]{Rudin:1976}. Therefore, an interval
  $[\tilde{t}_k,\tilde{t}_k+\delta]$ must have an infinite number of such points
  $t_k$ where $\rho$ is continuous. Additionally, for all $k=1,2,\ldots m$,
  we choose $t_k$ so that $t_{k-1}<\tilde{t}_k$.

  Next, we look at how $f_u$ and $f_l$ change compared to $\tilde{f}_u$ and
  $\tilde{f}_l$, respectively. We break interval $[\tilde{t}_{k-1},t_k]$ into
  three parts at $t_{k-1}$ and $\tilde{t}_k$. We notice that
$%
    [\tilde{t}_{k-1},\tilde{t}_k]\supset[t_{k-1},\tilde{t}_k]\subset[t_{k-1},t_k].
  $ %
    Due to the continuity of $f$ at points $\tilde{t}_{k-1}$ and
    $\tilde{t}_k$, we have
  \begin{align*}
    0              & \leq \sup\limits_{s\in[\tilde{t}_{k-1},\tilde{t}_k]}f(s)
                     -\sup\limits_{s\in[t_{k-1},\tilde{t}_k]}f(s) < \frac{\nu}{2},\\
    -\frac{\nu}{2} & <\sup\limits_{s\in[t_{k-1},\tilde{t}_k]}f(s)
                      -\sup\limits_{s\in[t_{k-1},t_k]}f(s) \leq 0,
  \end{align*}
  which we sum to get
  \begin{equation*}
    -\frac{\nu}{2}<\sup\limits_{s\in[\tilde{t}_{k-1},\tilde{t}_k]}f(s)
      -\sup\limits_{s\in[t_{k-1},t_k]}f(s)<\frac{\nu}{2}
      \Rightarrow\abs{\tilde{f}_u(\tilde{t}_k)-f_u(t_k)}<\frac{\nu}{2}.
  \end{equation*}
  Similarly, $\abs{\tilde{f}_l(\tilde{t}_k)-f_l(t_k)}<\frac{\nu}{2}$ and thus
$%
    \abs{\tilde{f}_u(\tilde{t}_k)-\tilde{f}_l(\tilde{t}_k)-(f_u(t_k)-f_l(t_k))}<\nu.
  $ %
  We mark now
  \begin{align*}
    \tilde{r}_k &= \tilde{f}_u(\tilde{t}_{k-1})-\tilde{f}_l(\tilde{t}_{k-1}),~
    r_k = f_u(t_{k-1})-f_l(t_{k-1}),~\\
    \tilde{\Delta}_k &= \rho(\tilde{t}_k)-\rho(\tilde{t}_{k-1}),~
    \Delta_k = \rho(t_k)-\rho(t_{k-1}),
  \end{align*}
  for which we have inequalities
$%
    \abs{\tilde{r}_k-r_k}<\nu$,
    $\abs{\tilde{\Delta}_k-\Delta_k}<\nu$,
    $\abs{\tilde{r}_k} \leq 2\sup |f|$, and
    $\abs{\Delta_k}\leq 2\,\sup \abs{\rho}.
  $ %
  By the triangle inequality, we have
  \begin{align*}
    \abs{\tilde{r}_k\,\tilde{\Delta}_k-r_k\,\Delta_k}
    &=
    \abs{\tilde{r}_k\,(\tilde{\Delta}_k-\Delta_k)+(\tilde{r}_k-r_k)\,\Delta_k}\\
    &\leq
    \abs{\tilde{r}_k\,(\tilde{\Delta}_k-\Delta_k)|+|(\tilde{r}_k-r_k)\,\Delta_k}
    \leq(\abs{\tilde{r}_k}+\abs{\Delta_k})\,\nu\leq\frac{\epsilon}{2\,n}.
  \end{align*}
  Because both $r_k\,\Delta_k$ and $\tilde{r}_k\,\tilde{\Delta}_k$ are
  positive, this means that
$%
    r_k\,\Delta_k\leq \tilde{r}_k\,\tilde{\Delta}_k+\frac{\epsilon}{2\,n}.
  $ %
  By \eqref{eq:discontinuous-point-limit}, we have
$  %
    \sum_{k=1}^n\tilde{r}_k\,\tilde{\Delta}_k<\frac{\epsilon}{2},
 $ %
  and therefore
  \begin{equation*}
    \sum_{k=1}^n(f_u(t_{k-1})-f_l(t_{k-1}))\,(\rho(t_k)-\rho(t_{k-1}))
    =\sum_{k=1}^n r_k\,\Delta_k
    \leq \sum_{k=1}^n\tilde{r}_k\,\tilde{\Delta}_k+\frac{\epsilon}{2}
    < \epsilon,
  \end{equation*}
  where none of the $t_k$ is a discontinuity point of $\rho$.
\end{proof}  
In the above proof, we used only the right continuity of $f$. However,
$f$ must be left continuous at points where $\rho$ is not left
continuous \cite[Theorem~7.29]{Apostol:1981}.

\end{appendices}

\bibliography{references}%

\end{document}